\documentclass[11pt,a4paper, leqno]{amsart}
\usepackage{amsmath,amsthm,amscd,amssymb,amsfonts, amsbsy}
\usepackage{latexsym}
\usepackage{txfonts}
\usepackage{exscale}

\usepackage[colorlinks=true, pdfstartview=FitV, linkcolor=blue, citecolor=red, urlcolor=blue]{hyperref} 

\calclayout
\allowdisplaybreaks

\theoremstyle{plain}
\newtheorem{theorem}[equation]{Theorem}
\newtheorem{lemma}[equation]{Lemma}
\newtheorem{corollary}[equation]{Corollary}
\newtheorem{proposition}[equation]{Proposition}

\theoremstyle{definition}
\newtheorem{definition}[equation]{Definition}

\theoremstyle{remark}
\newtheorem{remark}[equation]{Remark}

\numberwithin{equation}{section}

\newcommand{\eps}{\varepsilon}

\newcommand{\dint}{\int\!\!\!\int}

\newcommand{\dist}{\operatorname{dist}}

\newcommand{\tomt}{\tilde{\omega}_{\star}}

\newcommand{\re}{\mathbb{R}}

\newcommand{\ree}{\mathbb{R}^{n+1}}

\newcommand{\dd}{\mathbb{D}}
\newcommand{\A}{\mathbb{A}}

\newcommand{\F}{\mathcal{F}}

\newcommand{\W}{\mathcal{W}}
\newcommand{\T}{\mathcal{T}}

\newcommand{\pom}{\partial\Omega}

\newcommand{\hm}{\omega}

\renewcommand{\P}{\mathcal{P}}

\renewcommand{\emptyset}{\mbox{\textup{\O}}}

\DeclareMathOperator{\supp}{supp}

\DeclareMathOperator{\diam}{diam}
\DeclareMathOperator{\interior}{int}

\def\div{\mathop{\operatorname{div}}\nolimits}

\begin{document}
\allowdisplaybreaks

\title[Uniform rectifiability and harmonic measure II]{Uniform rectifiability and harmonic measure II: Poisson kernels in $L^p$ imply uniform rectifiability}

\author{Steve Hofmann}

\address{Steve Hofmann
\\
Department of Mathematics
\\
University of Missouri
\\
Columbia, MO 65211, USA} \email{hofmanns@missouri.edu}

\author{Jos\'{e} Mar\'{\i}a Martell}
\address{Jos\'{e} Mar\'{\i}a Martell
\\
Instituto de Ciencias Matem\'{a}ticas CSIC-UAM-UC3M-UCM
\\
Consejo Superior de Investigaciones Cient\'{\i}ficas
\\
C/ Nicol\'{a}s Cabrera, 13-15
\\
E-28049 Madrid, Spain} \email{chema.martell@icmat.es}

\author{Ignacio Uriarte-Tuero}

\address{Ignacio Uriarte-Tuero
\\
Department of Mathematics
\\
Michigan State University
\\
East Lansing, MI 48824, USA}
\email{ignacio@math.msu.edu}

\thanks{The first author was supported by NSF grant DMS-0801079.
The second author was supported by MINECO Grant MTM2010-16518 and ICMAT Severo Ochoa project SEV-2011-0087.
\\
\indent
This work has been possible thanks to the support and hospitality of the \textit{University of Missouri-Columbia} (USA),  the \textit{Consejo Superior de Investigaciones Cient\'{\i}ficas} (Spain), the \textit{Universidad Aut\'{o}noma de Madrid} (Spain), and the \textit{Australian National University} (Canberra, Australia). The first two authors would like to express their gratitude to these institutions.
}

\date{\today}


\begin{abstract}
We present the converse to a
higher dimensional, scale-invariant version of a classical theorem of F. and M. Riesz
\cite{Rfm}.
More precisely, for $n\geq 2$,
for an ADR domain $\Omega\subset \re^{n+1}$ which satisfies the Harnack Chain condition plus
an interior (but not exterior) Corkscrew condition, we show that absolute continuity
of harmonic measure with respect to surface measure on $\pom$, with scale invariant
higher integrability of the Poisson kernel, is sufficient to imply uniformly rectifiable  of $\pom$.
\end{abstract}

\maketitle

\tableofcontents

\section{Introduction}

This paper is a sequel to the work of the first two named authors \cite {HM-I}, in which we
have presented
a higher dimensional, scale invariant version of the classical theorem of F. and M. Riesz \cite{Rfm}.
The F. and M. Riesz Theorem states that for a
simply connected domain $\Omega$ in the complex plane, with a rectifiable boundary, one has that harmonic measure is absolutely continuous with respect to arclength measure on the boundary.
In \cite{HM-I}, we proved a scale invariant version of the latter result in higher dimensions.
That is, we showed that for $\Omega\subset \ree,\, n\geq 2$, satisfying certain quantitative topological
properties, whose boundary is rectifiable in an appropriate quantitative sense, one has that harmonic measure satisfies a scale invariant version of absolute continuity with respect to surface measure
(the so-called ``weak-$A_\infty$"
property;  cf. Definition \ref{def1.ainfty} below).   To be more precise, assuming that
$\Omega$ satisfies interior ``Corkscrew" and ``Harnack Chain" conditions (these
are scale invariant versions of the topological properties of openness and path connectedness;
cf. Definitions \ref{def1.cork} and \ref{def1.hc} below), and that $\partial \Omega$ is ``Uniformly Rectifiable"  (a quantitative, scale invariant version of rectifiability;  cf. Definition
\ref{def1.UR}), we showed that harmonic measure belongs to weak-$A_\infty$ with respect to surface measure on $\pom$.  Let us note that the weak-$A_\infty$ property implies that the Poisson kernel
(i.e., the Radon-Nikodym derivative of harmonic measure with respect to surface measure),
satisfies a scale-invariant higher integrability condition (cf. \eqref{eq1.lq}.)

In the present paper, we obtain a converse to the main result of \cite{HM-I}, that is,
we show that if $\Omega\subset \ree$ satisfies
interior Corkscrew and Harnack Chain conditions, if $\pom$ is $n$-dimensional
``Ahlfors-David Regular"  (cf.  Definition \ref{def1.ADR}), and if harmonic measure is absolutely continuous with respect to surface measure, with Poisson kernel satisfying the scale invariant higher integrability condition \eqref{eq1.lq}, then $\pom$ is Uniformly Rectifiable
(cf. Theorem \ref{theorem:converse}).   We observe that this result is in the spirit of
the solution of the Painlev\'{e} problem (\cite{T}, but see also \cite{Ch},  \cite{MMV},
\cite{D} and \cite {Vo}),
in which analytic information is used to establish rectifiability properties of a set, via the use of $Tb$ theory.  In our case, we use a so called ``local $Tb$" theorem, related to the technology of the solution of the Kato square root problem \cite{HMc}, \cite{HLMc}, \cite{AHLMcT}, and proved in \cite{GM}.

We observe that our work here and in \cite{HM-I}
may be viewed as a ``large constant'' analogue of the series of papers by
Kenig and Toro \cite{KT1,KT2,KT3}.   These papers say that in the presence of a Reifenberg flatness
condition and Ahlfors-David regularity,  $\log k \in VMO$ iff $\nu \in VMO$, where $k$ is the Poisson kernel with pole at some fixed point, and
$\nu$ is the unit normal to the boundary.  Moreover, given the same background hypotheses,
the condition that $\nu \in VMO$ is equivalent to
a uniform rectifiability (UR) condition with vanishing trace,
thus $\log k \in VMO \iff vanishing \,\, UR.$  On the other hand, our
large constant version ``almost'' says  ``$\,\log k \in BMO\iff UR\,$'',
given interior Corkscrews and Harnack Chains.
Indeed,  it is well known that the $A_\infty$ condition
(i.e.,  weak-$A_\infty$ plus the doubling property) implies that $\log k \in BMO$, while
if $\log k \in BMO$ with small norm, then $k\in A_\infty$.

We also point out that another antecedent of our main result here has appeared
in the work of Lewis and Vogel \cite{LV}.  The main emphasis of the latter pair of authors concerns
so called ``$p$-harmonic measure" (the analogue of harmonic measure corresponding to the $p$-Laplacian), but specializing their results to the classical Laplacian, they show that for a domain
$\Omega$, for which harmonic measure $\omega$
is an Ahlfors-David regular measure on $\partial \Omega$
(thus, the Poisson kernel $k =d\omega/d\sigma$ is bounded between two positive constants), then
$\partial \Omega$ is uniformly rectifiable.    The assumption that $k\approx 1$ is of course the strongest form of the higher integrability/(weak) reverse H\"{o}lder conditions that we consider here.

We refer the reader to the introduction of \cite{HM-I}, for a detailed historical survey
of related work.

\subsection{Notation and Definitions}

\begin{list}{$\bullet$}{\leftmargin=0.4cm  \itemsep=0.2cm}

\item We use the letters $c,C$ to denote harmless positive constants, not necessarily
the same at each occurrence, which depend only on dimension and the
constants appearing in the hypotheses of the theorems (which we refer to as the
``allowable parameters'').  We shall also
sometimes write $a\lesssim b$ and $a \approx b$ to mean, respectively,
that $a \leq C b$ and $0< c \leq a/b\leq C$, where the constants $c$ and $C$ are as above, unless
explicitly noted to the contrary.  At times, we shall designate by $M$ a particular constant whose value will remain unchanged throughout the proof of a given lemma or proposition, but
which may have a different value during the proof of a different lemma or proposition.

\item Given a domain $\Omega \subset \ree$, we shall
use lower case letters $x,y,z$, etc., to denote points on $\partial \Omega$, and capital letters
$X,Y,Z$, etc., to denote generic points in $\ree$ (especially those in $\ree\setminus \partial\Omega$).

\item The open $(n+1)$-dimensional Euclidean ball of radius $r$ will be denoted
$B(x,r)$ when the center $x$ lies on $\partial \Omega$, or $B(X,r)$ when the center
$X \in \ree\setminus \partial\Omega$.  A ``surface ball'' is denoted
$\Delta(x,r):= B(x,r) \cap\partial\Omega.$

\item Given a Euclidean ball $B$ or surface ball $\Delta$, its radius will be denoted
$r_B$ or $r_\Delta$, respectively.

\item Given a Euclidean or surface ball $B= B(X,r)$ or $\Delta = \Delta(x,r)$, its concentric
dilate by a factor of $\kappa >0$ will be denoted
by $\kappa B := B(X,\kappa r)$ or $\kappa \Delta := \Delta(x,\kappa r).$

\item For $X \in \ree$, we set $\delta(X):= \dist(X,\partial\Omega)$.

\item We let $H^n$ denote $n$-dimensional Hausdorff measure, and let
$\sigma := H^n\big|_{\partial\Omega}$ denote the ``surface measure'' on $\partial \Omega$.

\item For a Borel set $A\subset \ree$, we let $1_A$ denote the usual
indicator function of $A$, i.e. $1_A(x) = 1$ if $x\in A$, and $1_A(x)= 0$ if $x\notin A$.

\item For a Borel set $A\subset \ree$,  we let $\interior(A)$ denote the interior of $A$.
If $A\subset \partial\Omega$, then $\interior(A)$ will denote the relative interior, i.e., the largest relatively open set in $\partial\Omega$ contained in $A$.  Thus, for $A\subset \partial\Omega$,
the boundary is then well defined by $\partial A := \overline{A} \setminus {\rm int}(A)$.

\item For a Borel set $A$, we denote by $\mathcal{C}(A)$ the space of continuous functions on
$A$, by $\mathcal{C}_c(A)$ the subspace of $\mathcal{C}(A)$
with compact support in $A$, and by $\mathcal{C}_b(A)$ the
space of bounded continuous functions on $A$.  If $A$ is unbounded, we denote by
$C_0(A)$ the space of continuous functions on $A$ converging to $0$ at infinity.

\item For a Borel subset $A\subset\partial\Omega$, we
set $\fint_A f d\sigma := \sigma(A)^{-1} \int_A f d\sigma$.

\item We shall use the letter $I$ (and sometimes $J$)
to denote a closed $(n+1)$-dimensional Euclidean cube with sides
parallel to the co-ordinate axes, and we let $\ell(I)$ denote the side length of $I$.
We use $Q$ to denote a dyadic ``cube''
on $\partial \Omega$.  The
latter exist, given that $\partial \Omega$ is ADR  (cf. \cite{DS1}, \cite{Ch}), and enjoy certain properties
which we enumerate in Lemma \ref{lemmaCh} below.

\end{list}

\begin{definition} ({\bf Corkscrew condition}).  \label{def1.cork}
Following
\cite{JK}, we say that a domain $\Omega\subset \ree$
satisfies the ``Corkscrew condition'' if for some uniform constant $c>0$ and
for every surface ball $\Delta:=\Delta(x,r),$ with $x\in \partial\Omega$ and
$0<r<\diam(\partial\Omega)$, there is a ball
$B(X_\Delta,cr)\subset B(x,r)\cap\Omega$.  The point $X_\Delta\subset \Omega$ is called
a ``Corkscrew point'' relative to $\Delta.$  We note that  we may allow
$r<C\diam(\pom)$ for any fixed $C$, simply by adjusting the constant $c$.
\end{definition}

\begin{remark}\label{remark1.2}
We note that, on the other hand, every $X\in\Omega$, with $\delta(X)<\diam(\pom)$,
may be viewed as a Corkscrew point,
relative to some surface ball $\Delta\subset\pom$.
Indeed, set $r=K \delta(X)$,  with $K>1$, fix
$x\in\pom$ such that $|X-x|=\delta(X)$, and let
$\Delta:=\Delta(x,r)$.
\end{remark}

\begin{definition}({\bf Harnack Chain condition}).  \label{def1.hc} Again following \cite{JK}, we say
that $\Omega$ satisfies the Harnack Chain condition if there is a uniform constant $C$ such that
for every $\rho >0,\, \Lambda\geq 1$, and every pair of points
$X,X' \in \Omega$ with $\delta(X),\,\delta(X') \geq\rho$ and $|X-X'|<\Lambda\,\rho$, there is a chain of
open balls
$B_1,...,B_N \subset \Omega$, $N\leq C(\Lambda)$,
with $X\in B_1,\, X'\in B_N,$ $B_k\cap B_{k+1}\neq \emptyset$
and $C^{-1}\diam (B_k) \leq \dist (B_k,\partial\Omega)\leq C\diam (B_k).$  The chain of balls is called
a ``Harnack Chain''.
\end{definition}

We remark that the Corkscrew condition is a quantitative, scale invariant
version of the fact that $\Omega$ is open, and the Harnack Chain condition is a
scale invariant version of path connectedness.


\begin{definition}\label{def1.ADR}
({\bf Ahlfors-David regular}). We say that a closed set $E \subset \ree$ is $n$-dimensional ADR (or simply ADR) if
there is some uniform constant $C$ such that
\begin{equation} \label{eq1.ADR}
\frac1C\, r^n \leq H^n(E\cap B(x,r)) \leq C\, r^n,\,\,\,\forall r\in(0,R_0),x \in E,\end{equation}
where $R_0$ is the diameter of $E$ (which may be infinite).   When $E=\partial \Omega$,
the boundary of a domain $\Omega$, we shall sometimes for convenience simply
say that ``$\Omega$ has the ADR property'' to mean that $\partial \Omega$ is ADR.
\end{definition}

\begin{definition}\label{def1.UR}
({\bf Uniform Rectifiability}). Following
David and Semmes \cite{DS1,DS2}, we say
that a closed set $E\subset \ree$ is    $n$-dimensional UR (or simply UR) (``Uniformly Rectifiable''), if
it satisfies the ADR condition \eqref{eq1.ADR}, and if for some uniform constant $C$ and for every
Euclidean ball $B:=B(x_0,r), \, r\leq \diam(E),$ centered at any point $x_0 \in E$, we have the Carleson measure estimate
\begin{equation}\label{eq1.sf}\dint_{B}
|\nabla^2 \mathcal{S}1(X)|^2 \,\dist(X,E) \,dX \leq C r^n,
\end{equation}
where $\mathcal{S}f$ is the single layer potential of $f$, i.e.,
\begin{equation}\label{eq1.layer}
\mathcal{S}f(X) :=c_n\, \int_{E} |X-y|^{1-n} f(y) \,dH^n(y).
\end{equation}
Here, the normalizing constant $c_n$ is chosen so that
$\mathcal{E}(X) := c_n |X|^{1-n}$ is the usual fundamental solution for the Laplacian
in $\ree.$  When $E=\partial \Omega$,
the boundary of a domain $\Omega$, we shall sometimes for convenience simply
say that ``$\Omega$ has the UR property'' to mean that $\partial \Omega$ is UR.
\end{definition}

We note that there are numerous characterizations of uniform rectifiability
given in \cite{DS1,DS2}. Let us note that, by ``$T1$ reasoning'', the Carleson measure condition \eqref{eq1.sf} is
equivalent to the global $L^2$ bound
\begin{equation}\label{eqA.6a}
\dint_{\ree}|\nabla^2\mathcal{S}f (X)|^2\,\delta(X)\, dX
\,\leq\, C\,\|f\|^2_{L^2(\partial\Omega)}.
\end{equation}
The condition \eqref{eq1.sf} will be most useful for our purposes, and
appears in \cite[Chapter 3,  Part III]{DS2}.
We remark that the UR sets
are precisely those for which all ``sufficiently nice'' singular integrals
are bounded on $L^2$ (see \cite{DS1}).

\begin{definition}\label{def1.bp} {\bf (``Big Pieces'')}.
Given  a closed set $E\subset\ree$ such that $E$ is $n$-dimensional ADR, and a collection $\mathcal{S}$ of domains in $\ree$, we say that
$E$ has ``big pieces of boundaries of $\mathcal{S}\,$'' (denoted
$E\in BP(\partial\mathcal{S})$)
if there is a constant $0<\alpha\le 1$ such that for
every $x\in E$, and $0<r<\diam(E)$, there is a domain $\Omega' \in \mathcal{S}$ such that
$$
H^n(\pom'\cap B(x,r)\cap E) \geq \alpha H^n(B(x,r)\cap E) \approx \alpha r^n.
$$
\end{definition}

\begin{lemma}\label{lemmaCh}({\bf Existence and properties of the ``dyadic grid''})
\cite{DS1,DS2}, \cite{Ch}.
Suppose that $E\subset \ree$ satisfies
the ADR condition \eqref{eq1.ADR}.  Then there exist
constants $ a_0>0,\, \eta>0$ and $C_1<\infty$, depending only on dimension and the
ADR constants, such that for each $k \in \mathbb{Z},$
there is a collection of Borel sets (``cubes'')
$$
\mathbb{D}_k:=\{Q_{j}^k\subset E: j\in \mathfrak{I}_k\},$$ where
$\mathfrak{I}_k$ denotes some (possibly finite) index set depending on $k$, satisfying

\begin{list}{$(\theenumi)$}{\usecounter{enumi}\leftmargin=.8cm
\labelwidth=.8cm\itemsep=0.2cm\topsep=.1cm
\renewcommand{\theenumi}{\roman{enumi}}}

\item $E=\cup_{j}Q_{j}^k\,\,$ for each
$k\in{\mathbb Z}$

\item If $m\geq k$ then either $Q_{i}^{m}\subset Q_{j}^{k}$ or
$Q_{i}^{m}\cap Q_{j}^{k}=\emptyset$.

\item For each $(j,k)$ and each $m<k$, there is a unique
$m$ such that $Q_{j}^k\subset Q_{i}^m$.

\item Diameter $\left(Q_{j}^k\right)\leq C_12^{-k}$.

\item Each $Q_{j}^k$ contains some ``surface ball'' $\Delta \big(x^k_{j},a_02^{-k}\big):=
B\big(x^k_{j},a_02^{-k}\big)\cap E$.

\item $H^n\left(\left\{x\in Q^k_j:{\rm dist}(x,E\setminus Q^k_j)\leq \tau \,2^{-k}\right\}\right)\leq
C_1\,\tau^\eta\,H^n\left(Q^k_j\right),$ for all $k,j$ and for all $\tau\in (0,a_0)$.
\end{list}
\end{lemma}

A few remarks are in order concerning this lemma.

\begin{list}{$\bullet$}{\leftmargin=0.4cm  \itemsep=0.2cm}

\item In the setting of a general space of homogeneous type, this lemma has been proved by Christ
\cite{Ch}.  In that setting, the
dyadic parameter $1/2$ should be replaced by some constant $\delta \in (0,1)$.
It is a routine matter to verify that one may take $\delta = 1/2$ in the presence of the Ahlfors-David
property (\ref{eq1.ADR}) (in this more restrictive context, the result already appears in \cite{DS1,DS2}).

\item  For our purposes, we may ignore those
$k\in \mathbb{Z}$ such that $2^{-k} \gtrsim {\rm diam}(E)$, in the case that the latter is finite.

\item  We shall denote by  $\mathbb{D}=\mathbb{D}(E)$ the collection of all relevant
$Q^k_j$, i.e., $$\mathbb{D} := \cup_{k} \mathbb{D}_k,$$
where, if $\diam (E)$ is finite, the union runs
over those $k$ such that $2^{-k} \lesssim  {\rm diam}(E)$.

\item Properties $(iv)$ and $(v)$ imply that for each cube $Q\in\mathbb{D}_k$,
there is a point $x_Q\in E$, a Euclidean ball $B(x_Q,r)$ and a surface ball
$\Delta(x_Q,r):= B(x_Q,r)\cap E$ such that
$r\approx 2^{-k} \approx {\rm diam}(Q)$
and \begin{equation}\label{cube-ball}
\Delta(x_Q,r)\subset Q \subset \Delta(x_Q,Cr),\end{equation}
for some uniform constant $C$.
We shall denote this ball and surface ball by
\begin{equation}\label{cube-ball2}
B_Q:= B(x_Q,r) \,,\qquad\Delta_Q:= \Delta(x_Q,r),\end{equation}
and we shall refer to the point $x_Q$ as the ``center'' of $Q$.

\item Let us now specialize to the case that  $E=\pom$, with $\Omega$ satisfying the Corkscrew condition.  Given $Q\in \mathbb{D}(\partial\Omega)$,
we
shall sometimes refer to a ``Corkscrew point relative to $Q$'', which we denote by
$X_Q$, and which we define to be the corkscrew point $X_\Delta$ relative to the ball
$\Delta:=\Delta_Q$ (cf. \eqref{cube-ball}, \eqref{cube-ball2} and Definition \ref{def1.cork}).  We note that
\begin{equation}\label{eq1.cork}
\delta(X_Q) \approx \dist(X_Q,Q) \approx \diam(Q).
\end{equation}
\end{list}

\begin{list}{$\bullet$}{\leftmargin=0.4cm  \itemsep=0.2cm}

\item For a dyadic cube $Q\in \mathbb{D}_k$, we shall
set $\ell(Q) = 2^{-k}$, and we shall refer to this quantity as the ``length''
of $Q$.  Evidently, $\ell(Q)\approx \diam(Q).$

\item For a dyadic cube $Q \in \mathbb{D}$, we let $k(Q)$ denote the ``dyadic generation''
to which $Q$ belongs, i.e., we set  $k = k(Q)$ if
$Q\in \mathbb{D}_k$; thus, $\ell(Q) =2^{-k(Q)}$.

\end{list}

\begin{definition} ({\bf $A_\infty$, $A_\infty^{\rm dyadic}$ and weak-$A_\infty$}).  \label{def1.ainfty}
Given a surface ball
$\Delta= B\cap\pom$, a Borel measure $\hm$ defined on $\partial \Omega$
is said to belong to the class $A_\infty(\Delta)$ if there are positive constants $C$ and $\theta$
such that for every $\Delta'=B'\cap\pom$ with $B'\subseteq B$,
\and every Borel set $F\subset\Delta'$, we have
\begin{equation}\label{eq1.ainfty}
\hm (F)\leq C \left(\frac{\sigma(F)}{\sigma(\Delta')}\right)^\theta\,\hm (\Delta').
\end{equation}
If we replace the surface balls $\Delta$
and $\Delta'$  by a dyadic cube $Q$
and its dyadic subcubes $Q'$, with $F\subset Q'$,
then we say that $\hm\in A_\infty^{\rm dyadic}(Q)$:
\begin{equation}\label{eq1.ainftydyadic}
\hm (F)\leq C \left(\frac{\sigma(F)}{\sigma(Q')}\right)^\theta\,\hm (Q').
\end{equation}
Similarly, $\hm \in$ weak-$A_\infty(\Delta)$, with $\Delta =B\cap\pom$, if for every
$\Delta'=B'\cap\pom$ with $2B'\subseteq B$,
we have
\begin{equation}\label{eq1.wainfty}
\hm (F) \leq C \left(\frac{\sigma(F)}{\sigma(\Delta')}\right)^\theta\,\hm (2 \Delta')
\end{equation}
\end{definition}
As is well known  \cite{CF}, \cite{GR}, \cite{Sa},
the $A_\infty$ (resp. weak-$A_\infty$) condition is equivalent to
the property that the measure
$\hm$ is absolutely continuous with respect to $\sigma$, and that its density
satisfies a reverse H\"{o}lder (resp. weak reverse H\"{o}lder) condition.
In this paper, we are interested in the case that $\hm = \omega^{X}$,
the harmonic measure with
pole at $X$.  In that setting, we let
$k^X:= d\omega^X/d\sigma$ denote
the Poisson kernel, so that \eqref{eq1.ainfty} is equivalent to the reverse H\"{o}lder estimate
\begin{equation}\label{eq1.RH}
\left(\fint_{\Delta'} \left(k^{X}\right)^q d\sigma\right)^{1/q} \leq C\fint_{\Delta'}k^{X}\,d\sigma
\,,
\end{equation}
for some $q>1$
and for some uniform constant $C$.
In particular, when $\Delta'=\Delta$, and $X= X_\Delta$, a Corkscrew point relative to $\Delta$,
the latter estimate reduces to
\begin{equation}\label{eq1.lq}
\int_{\Delta} \left(k^{X_\Delta}\right)^q d\sigma \leq C\, \sigma(\Delta)^{1-q}.
\end{equation}
Similarly, \eqref{eq1.wainfty} is equivalent to
\begin{equation}\label{eq1.WRH}
\left(\fint_{\Delta'} \left(k^{X}\right)^q d\sigma\right)^{1/q} \leq
C\fint_{2\Delta'}k^{X}\,d\sigma
\,.
\end{equation}
Assuming that the latter bound holds with $\Delta'=\Delta$, and with $X=X_\Delta$,
then one again obtains \eqref{eq1.lq}.

\subsection{Statement of the Main Result}

Our main result is as follows. We shall use the terminology that a connected open set
$\Omega\subset \ree$ is a {\bf 1-sided} NTA domain if it satisfies {\bf interior} (but not necessarily
exterior) Corkscrew and Harnack Chain conditions.

\begin{theorem}\label{theorem:converse}
Let $\Omega \subset \ree,\,n\geq2,$
be a 1-sided NTA domain, whose boundary
is $n$-dimensional ADR. Suppose also that  harmonic measure
$\omega$ is absolutely continuous with respect to surface measure
and that the Poisson kernel $k=d\omega/d\sigma$ satisfies the scale invariant  estimate
\begin{equation}\label{eqn:scale-inva:converse}
\int_{\Delta} \left(k^{X_\Delta}\right)^p d\sigma \leq C\, \sigma(\Delta)^{1-p},
\end{equation}
for some $1<p<\infty$ and for all surface balls $\Delta$. Then $\partial \Omega$ is UR.
\end{theorem}

\begin{remark}
As mentioned above \eqref{eqn:scale-inva:converse} is (apparently) weaker than $\omega^{X_\Delta}$ being in weak-$A_\infty(\Delta)$. However, \textit{a posteriori}, we obtain that these two conditions are equivalent. Namely, Theorem \ref{theorem:converse} shows that $\partial \Omega$ is UR and therefore we can apply \cite[Theorem 1.26]{HM-I} to obtain that $\omega^{X_\Delta}$ belongs to weak-$A_\infty(\Delta)$.
\end{remark}

Theorem \ref{theorem:converse} leads to an immediate ``self-improvement'' of itself,
in which the hypotheses are assumed to hold only in an appropriate ``big pieces'' sense.

\begin{theorem}\label{theorem:converse-BP}
Let $E\subset\ree$ be a closed set and assume that $E$ is $n$-dimensional ADR.
Suppose further that
$E \in BP(\partial\mathcal{S})$ (cf. Definition \ref{def1.bp}), where
$\mathcal{S}$ is a collection of 1-sided NTA domains with $n$-dimensional ADR boundaries,
with uniform control of all of the relevant Corkscrew, Harnack Chain and ADR constants. Assume also that there exist $1<p<\infty$ and $C\ge 1$ such that for every $\Omega'\in \mathcal{S}$ we have that harmonic measure $\omega_{\Omega'}$ is absolutely continuous with respect to surface measure
and that the Poisson kernel $k_{\Omega'}=d\omega_{\Omega'}/d\sigma_{\Omega'}$ satisfies the scale invariant  estimate
\begin{equation}\label{eqn:scale-inva:converse:BP}
\int_{\Delta'} \left(k_{\Omega'}^{X_{\Delta'}}\right)^p d\sigma_{\Omega'} \leq C\, \sigma_{\Omega'}(\Delta')^{1-p},
\end{equation}
and for all surface balls $\Delta'$ on  $\Omega'$. Then $E$ is UR.
\end{theorem}

The proof of this result is as follows. Let $x\in E$ and $0<r<\diam(E)$. Under the hypotheses of Theorem
\ref{theorem:converse-BP}, there is  $\Omega'$ satisfying
the hypotheses of Theorem \ref{theorem:converse}, with the property that
for some $0<\alpha\le 1$, we have
\begin{equation}\label{eqn-bp-contact}
\sigma\left(\pom' \cap B(x,r)\cap E\right)\,\geq\, \alpha\, \sigma(B(x,r)\cap E).
\end{equation}
By Theorem \ref{theorem:converse}, \eqref{eqn:scale-inva:converse:BP} implies that $\pom'$ is UR with uniform control of the constants (since the ADR and 1-sided NTA constants are uniformly controlled, and also $p$ and $C$ in \eqref{eqn:scale-inva:converse:BP} are independent of $\Omega'$).  This and \eqref{eqn-bp-contact} implies that $E$ has big pieces of UR sets and therefore $E$ is UR, see
\cite{DS2}.

\section{Proof of Theorem \ref{theorem:converse}}

\subsection{Preliminaries}

We collect some of the definition and auxiliary results from \cite{HM-I} that will be used later.
In the sequel, $\Omega \subset \ree,\, n\geq 2,$ will be a connected, open set, $\omega^X$ will denote
harmonic measure for $\Omega$, with pole at $X$, and $G(X,Y)$
will be the Green function. At least in the case that $\Omega$ is bounded, we may, as usual,
define $\omega^X$ via the maximum principle and the Riesz representation theorem,
after first using the method of Perron (see, e.g., \cite[pp. 24--25]{GT}) to construct a harmonic function ``associated'' to arbitrary continuous boundary
data.\footnote{Since we have made no assumption as regards
Wiener's regularity criterion, our harmonic function is a generalized solution, which
may not be continuous up to the boundary.}  For unbounded $\Omega$, we may still
define harmonic measure via a standard approximation scheme, see \cite[Section 3]{HM-I} for more details.  We note for future reference that $\omega^X$
is a non-negative, finite, outer regular Borel measure.


The Green function may now be constructed
by setting
\begin{equation}\label{eq2.greendef}
G(X,Y):= \mathcal{E}\,(X-Y) - \int_{\partial\Omega}\mathcal{E}\,(X-z)\,d\omega^Y(z),
\end{equation}
where $\mathcal{E}\,(X):= c_n |X|^{1-n}$ is the usual fundamental solution for the Laplacian
in $\ree$.  We choose the normalization that makes $\mathcal{E}$ positive.

\begin{lemma}[Bourgain \cite{B}]\label{Bourgainhm}  Suppose that
$\partial \Omega$ is $n$-dimensional ADR.  Then there are uniform constants $c\in(0,1)$
and $C\in (1,\infty)$,
such that for every $x \in \partial\Omega$, and every $r\in (0,\diam(\partial\Omega))$,
if $Y \in \Omega \cap B(x,cr),$ then
\begin{equation}\label{eq2.Bourgain1}
\omega^{Y} (\Delta(x,r)) \geq 1/C>0 \;.
\end{equation}
In particular, if $\Omega$ satisfies the Corkscrew and Harnack Chain conditions,
then for every surface ball $\Delta$, we have
\begin{equation}\label{eq2.Bourgain2}
\omega^{X_\Delta} (\Delta) \geq 1/C>0 \;.
\end{equation}
\end{lemma}

We next introduce some terminology.
\begin{definition}\label{qualextCS}
A domain $\Omega$ satisfies the \textbf{qualitative exterior
Corkscrew} condition if there exists $N\gg 1$ such that $\Omega$ has exterior corkscrew points at all scales smaller than $2^{-N}$. That is, there exists a constant $c_N$ such that for every surface ball $\Delta=\Delta(x,r)$, with $x\in \pom$  and $r\le 2^{-N}$, there is a ball $B(X_\Delta^{ext},c_N\,r)\subset B(x,r)\cap\Omega_{ext}$. \end{definition}
Let us observe that if
$\Omega$ satisfies the qualitative exterior
Corkscrew condition, then every point in $\pom$ is regular in the sense of Wiener.
Moreover, for $1$-sided NTA domains, the qualitative
exterior Corkscrew points allow local H\"{o}lder continuity
at the boundary (albeit with bounds which may depend badly on $N$).

\begin{lemma}[{\cite[Lemma 3.11]{HM-I}}]\label{lemma2.green} There are positive, finite constants $C$, depending only on dimension,
and $c(n,\theta)$, depending on dimension and $\theta \in (0,1),$
such that the Green function satisfies
\begin{eqnarray}\label{eq2.green}
&G(X,Y) \leq C\,|X-Y|^{1-n}\\[4pt]\label{eq2.green2}
& c(n,\theta)\,|X-Y|^{1-n}\leq G(X,Y)\,,\quad {\rm if } \,\,\,|X-Y|\leq \theta\, \delta(X)\,, \,\, \theta \in (0,1)\,.
\end{eqnarray}
Moreover, if every point on $\pom$  is regular in the sense of Wiener, then
\begin{equation}\label{eq2.green3}
G(X,Y)\geq 0
\end{equation}
\begin{equation}\label{eq2.green4}
G(X,Y)=G(Y,X)\,,\qquad \forall X,Y\in\Omega\,,\, X\neq Y\,.
\end{equation}
\end{lemma}

\begin{lemma}[{\cite[Lemma 3.30]{HM-I}}] \label{lemma2.cfms}
Let $\Omega$ be a 1-sided NTA domain with $n$-dimensional ADR boundary, and suppose that every $x \in \partial \Omega$ is regular in the sense of Wiener. Fix
$B_0:=B(x_0,r_0)$  with $x_0\in\pom$, and $\Delta_0 := B_0\cap\partial\Omega$.
Let $B:=B(x,r)$,  $x\in \pom$, and
$\Delta:=B\cap\pom$, and suppose that $2B\subset B_0.$  Then for
$X\in \Omega\setminus B_0$ we have
\begin{equation}\label{eq2.CFMS1}
r^{n-1}G(X_\Delta,X) \leq C \omega^X(\Delta).
\end{equation}
If, in addition,  $\Omega$ satisfies the qualitative exterior corkscrew condition, then
\begin{equation}\label{eq2.CFMS2}
\omega^X(\Delta)\leq C r^{n-1}G(X_\Delta,X).
\end{equation}
The constants in \eqref{eq2.CFMS1} and \eqref{eq2.CFMS2} depend {\bf only} on
dimension and on the constants in the $ADR$ and 1-sided NTA
conditions.
\end{lemma}

\begin{corollary}[{\cite[Corollary 3.36]{HM-I}}]\label{cor2.double}
Suppose that $\Omega$ is a  1-sided NTA domain with $n$-dimensional ADR boundary and that it also satisfies the qualitative exterior
Corkscrew condition.  Let $B:=B(x,r)$,  $x\in \pom$,
$\Delta:= B\cap\partial\Omega$ and $X\in \Omega\setminus 4B.$
Then there is a uniform constant $C$ such that
$$\omega^X(2\Delta)\leq C\omega^X(\Delta).$$
\end{corollary}

We next introduce some ``discretized'' and ``geometric'' sawtooth and Carleson regions from \cite[Section 3]{HM-I}.
Given a ``dyadic cube''
$Q\in \dd(\partial\Omega)$, the {\bf discretized Carleson region} $\dd_Q$ is defined to be
\begin{equation}\label{eq2.discretecarl}
\dd_Q:= \left\{Q'\in \dd: Q'\subseteq Q\right\}.
\end{equation}
Given a family $\mathcal{F}$ of disjoint cubes $\{Q_j\}\subset \mathbb{D}$, we define
the {\bf global discretized sawtooth} relative to $\F$ by
\begin{equation}\label{eq2.discretesawtooth1}
\dd_{\F}:=\dd\setminus \bigcup_{\F} \dd_{Q_j}\,,
\end{equation}
i.e., $\dd_{\F}$ is the collection of all $Q\in\dd$ that are not contained in any $Q_j\in\F$.
Given some fixed cube $Q$,
the {\bf local discretized sawtooth} relative to $\F$ by
\begin{equation}\label{eq2.discretesawtooth2}
\dd_{\F,Q}:=\dd_Q\setminus \bigcup_{\F} \dd_{Q_j}=\dd_\F\cap\dd_Q.
\end{equation}

We also introduce the ``geometric'' Carleson regions and sawtooths.
Let us first recall that we write $k=k(Q)$ if $Q\in \mathbb{D}_k$
(cf. Lemma \ref{lemmaCh}), and in that case the ``length'' of $Q$ is denoted
by $\ell(Q)=2^{-k(Q)}$.
We also recall that there is a Corkscrew point $X_Q$, relative to each $Q\in\dd$
(in fact, there are many such, but we just pick one).
Let $\mathcal{W}=\W(\Omega)$ denote a collection
of (closed) dyadic Whitney cubes of $\Omega$, so that the cubes in $\mathcal{W}$
form a pairwise non-overlapping covering of $\Omega$, which satisfy
\begin{equation}\label{eq4.1}
4 \diam(I)\leq
\dist(4I,\partial\Omega)\leq \dist(I,\pom) \leq 40\diam(I)\,,\qquad \forall\, I\in \mathcal{W}\,\end{equation}
and also
$$(1/4)\diam(I_1)\leq\diam(I_2)\leq 4\diam(I_1)\,,$$
whenever $I_1$ and $I_2$ touch.
Let $\ell(I)$ denote the side length of $I$,
and write $k=k_I$ if $\ell(I) = 2^{-k}$.
There are $C_0\ge 1000\sqrt{n}$ and $m_0\ge 0$ large enough (depending only on the constants in the Corkscrew condition and in the
dyadic cube construction) so that for every cube $Q\in\dd$
\begin{equation}\label{eq2.whitney1}
\mathcal{W}_Q := \left\{I\in \mathcal{W}: \,k(Q) -m_0\,\leq \,k_I\,\leq \,k(Q)+1\,,\,{\rm and}
\dist(I,Q)\leq C_0\, 2^{-k(Q)}\right\}\,,
\end{equation}
satisfies that $X_Q\in I$ for some $I\in\W_Q$,
and for each dyadic child $Q^{j}$ of $Q$,
the respective Corkscrew points $X_{Q^{j}}\in I^{j}$ for some $I^{j}\in\W_Q$.
Moreover, we may always find an $I\in \W_Q$ with the slightly more precise property
that $k(Q)-1\leq k_I\leq k(Q)$ and
\begin{equation}\label{eq2.whitney2**}
\W_{Q_1}\cap \W_{Q_2}\neq\emptyset\,,\,\, {\rm whenever}\,\,1\leq
\frac{\ell(Q_2)}{\ell(Q_1)}\leq 2\,,\,\, {\rm and}\,\dist(Q_1,Q_2)\leq1000\ell(Q_2)\,.
\end{equation}

We introduce some notation:  given a subset
$A\subset \Omega$, we write $ X\rightarrow_A Y$
if the interior of $A$ contains all the balls in
a Harnack Chain (in $\Omega$), connecting $X$ to $Y$, and if, moreover,
for any point $Z$ contained
in any ball in the Harnack Chain, we have $
\dist(Z,\pom) \approx \dist(Z,\Omega\setminus A)\,,
$
with uniform control of the implicit constants.   We denote by $X(I)$ the center
of a cube $I\in \ree$, and we recall that $X_Q$ denotes a designated Corkscrew point relative
to $Q$ that can be assumed to be the center of some Whitney cube
$I$ such that $I\subset B_Q\cap\Omega$ and $\ell(I)\approx\ell(Q)\approx \dist(I,Q)$.

For each $I\in \W_Q$, we form a Harnack Chain, call it $H(I)$,
from the center $X(I)$ to the Corkscrew point $X_Q$.   We now denote by $\W(I)$
the collection of all Whitney cubes which meet at least one ball in the chain $H(I)$,
and we set
$$\W_Q^* := \bigcup_{I\in \W_Q} \W(I).$$
We also define, for $\lambda\in(0,1)$ to be chosen momentarily,
\begin{equation}\label{eq2.whitney3}
U_Q := \bigcup_{\mathcal{W}^*_Q} (1+\lambda)I=:
\bigcup_{I\in\,\mathcal{W}^*_Q} I^*\,.
\end{equation}
By construction, we then have that
\begin{equation}\label{eq2.whitney2*}
 \mathcal{W}_Q\subset \mathcal{W}^*_Q\subset\mathcal{W}\,\quad {\rm and}\quad
X_Q\in U_Q\,,\,\, X_{Q^j} \in U_Q\,,
\end{equation}
for each child $Q^j$ of $Q$.
It is also clear that there are uniform constants $k^*$ and $K_0$ such that
\begin{eqnarray}\label{eq2.whitney2}
& k(Q)-k^*\leq k_I \leq k(Q) + k^*\,, \quad \forall I\in \mathcal{W}^*_Q\\\nonumber
&X(I) \rightarrow_{U_Q} X_Q\,,\quad\forall I\in \mathcal{W}^*_Q\\ \nonumber
&\dist(I,Q)\leq K_0\,2^{-k(Q)}\,, \quad \forall I\in \mathcal{W}^*_Q\,,
\end{eqnarray}
where $k^*$, $K_0$ and the implicit constants in the condition $X(I)\to_{U_Q} X_Q$,
depend only  on the ``allowable
parameters'' (since $m_0$ and $C_0$ also have such dependence)
and on $\lambda$.
Thus, by the addition of a few nearby
Whitney cubes of diameter also comparable to that of $Q$,
we can ``augment'' $ \mathcal{W}_Q$ so that the Harnack Chain condition
holds in $U_Q$.

We fix the parameter $\lambda$ so that
for any $I,J\in\W$,
\begin{equation}\label{eq2.29*}
\begin{split}
\dist(I^*,J^*)& \approx \dist(I,J) \\
\interior(I^*)\cap\interior(J^*)\neq \emptyset& \iff \partial I \cap\partial J \neq \emptyset
\end{split}
\end{equation}
(the fattening thus ensures overlap of $I^*$ and $J^*$ for
any pair $I,J \in\W$ whose boundaries touch, so that the
Harnack Chain property then holds locally, with constants depending upon $\lambda$, in $I^*\cup J^*$).  By choosing $\lambda$ sufficiently small,
we may also suppose that there is a $\tau\in(1/2,1)$ such that for distinct $I,J\in\W$,
\begin{equation}\label{eq2.30*}
\tau J\cap I^* = \emptyset\,.
\end{equation}

\begin{remark}\label{remark_II.2.25}
We note that any sufficiently small choice of $\lambda$ (say $0<\lambda\le \lambda_0$) will do for our purposes.
\end{remark}

Of course, there may be some flexibility in the choice of additional Whitney cubes
which we add to form the augmented collection $\mathcal{W}^*_Q$, but having made such a choice
for each $Q \in \mathbb{D}$, we fix it for all time.

We may then define  the {\bf Carleson box}
associated to $Q$ by
\begin{equation}\label{eq2.box}
T_Q:={\rm int}\left( \bigcup_{Q'\in \dd_Q} U_{Q'}\right).
\end{equation}
Similarly, we may define geometric sawtooth regions as follows.
As above, give a family $\mathcal{F}$ of disjoint cubes $\{Q_j\}\subset \mathbb{D}$,
we define the {\bf global sawtooth} relative to $\mathcal{F}$ by
\begin{equation}\label{eq2.sawtooth1}
\Omega_{\mathcal{F}}:= {\rm int } \left( \bigcup_{Q'\in\dd_\F} U_{Q'}\right)\,,
\end{equation}
and again given some fixed $Q\in \dd$,
the {\bf local sawtooth} relative to $\mathcal{F}$ by
\begin{equation}\label{eq2.sawtooth2}
\Omega_{\mathcal{F},Q}:=  {\rm int } \left( \bigcup_{Q'\in\dd_{\F,Q}} U_{Q'}\right)\,.
\end{equation}
We also define as follows the ``Carleson box'' $T_\Delta$
associated to a surface ball $\Delta:=\Delta(x_\Delta,r)$.
Let $k(\Delta)$ denote the unique $k \in \mathbb{Z}$ such that $2^{-k-1}<200\,r\leq 2^{-k}$,
and set
\begin{equation}\label{eq2.box3}
\mathbb{D}^\Delta:= \{Q \in \mathbb{D}_{k(\Delta)}: Q\cap2\Delta \neq \emptyset\}.
\end{equation}
We then define
\begin{equation}\label{eq2.box2}
T_\Delta :={\rm int}\left(\bigcup_{Q\in \mathbb{D}^\Delta} \overline{T_Q}\right).
\end{equation}
Given a
surface ball $\Delta:=\Delta(x,r)$, let $B_\Delta := B(x,r)$, so that
$\Delta =B_\Delta\cap\partial\Omega$.
Then,
\begin{equation}\label{eq2.ball-box}
\frac54 B_\Delta\cap\Omega\subset T_\Delta.
\end{equation}
and
there exists $\kappa_0$ large enough such that
\begin{equation}\label{eq2.contain}
\overline{T_\Delta}\subset \kappa_0B_\Delta\cap\overline{\Omega}.\end{equation}

\medskip

\begin{lemma}[{\cite[Lemma 3.61]{HM-I}}]\label{lemma2.30}  Suppose that $\Omega$ is a  1-sided NTA domain
with an ADR boundary.
Then all of its  Carleson boxes
$T_Q$ and $T_\Delta$, and sawtooth regions
$ \Omega_{\mathcal{F}}$, and $\Omega_{\mathcal{F},Q}$ are also 1-sided NTA domains
with ADR boundaries.
 In all cases, the implicit constants are uniform, and
depend only on dimension and on the corresponding constants for $\Omega$.
\end{lemma}

\medskip

\begin{lemma}[{\cite[Lemma 3.62]{HM-I}}]\label{lemma:inher-quali}
Suppose that $\Omega$ is a 1-sided NTA domain with an ADR boundary and that $\Omega$ also satisfies the qualitative exterior
Corkscrew condition. Then all of its  Carleson boxes
$T_Q$ and $T_\Delta$, and sawtooth regions $ \Omega_{\mathcal{F}}$, and $\Omega_{\mathcal{F},Q}$ satisfy the qualitative exterior
Corkscrew condition. In all cases, the implicit constants are uniform, and
depend only on dimension and on the corresponding constants for $\Omega$.
\end{lemma}

\begin{corollary}[{\cite[Corollary 3.69]{HM-I}}]\label{cor2.poles}
Suppose that $\Omega$ is a  1-sided NTA domain with $n$-dimensional ADR boundary and that it also satisfies the qualitative exterior
Corkscrew condition. There is a uniform constant $C$ such that for every pair of surface balls
$\Delta:=B\cap\partial\Omega$,
and $\Delta':= B'\cap\partial\Omega$,  with $B'\subseteq B,$ and
for every $X\in \Omega\setminus 2\kappa_0B,$
where $\kappa_0$ is the constant in \eqref{eq2.contain} below, we have
$$\frac1C\, \omega^{X_\Delta}(\Delta') \,\leq\, \frac{\omega^X(\Delta')}{\omega^X(\Delta)}\,
\leq \,C\, \omega^{X_\Delta}(\Delta').$$
\end{corollary}

\begin{remark}\label{remark:lambda0}  Let us recall that the dilation factor $\lambda$, defining the fattened Whitney boxes $I^*$ and Whitney regions $U_Q$,
is allowed to be any fixed positive number no larger than some small $\lambda_0$
(cf. \eqref {eq2.whitney3} and Remark \ref{remark_II.2.25}).
For the rest of the paper we now fix $0<\lambda\le \lambda_0/2$, so that, in particular, the previous results apply not only to the Carleson boxes and sawtooths corresponding to $U_Q$ as defined above,  but also to those corresponding to ``fatter" Whitney regions
$U_Q^*=\cup_{\mathcal{W}^*_Q} I^{**}$ with $I^{**}=(1+2\,\lambda)\,I$.
We will work with these fatter regions in Sections \ref{section:good-lambda}
and \ref{section:proof-S} below.
\end{remark}

\subsection{Step 1: Passing to the approximating domains}

We first observe that without loss of generality we can assume that $p\le 2$: If $p>2$, \eqref{eqn:scale-inva:converse} and H\"{o}lder's inequality imply the same estimate with $p=2$.

We define approximating domains as follows.
For each large integer $N$,  set $\F_N := \dd_N$.  We then let
$\Omega_N := \Omega_{\F_N}$ denote the usual (global) sawtooth with respect to the family
$\F_N$ (cf. \eqref{eq2.whitney2}, \eqref{eq2.whitney3} and \eqref{eq2.sawtooth1}.)  Thus,
\begin{equation}\label{eq7.on}
\Omega_N =\interior\left(\bigcup_{Q\in \dd:\,\ell(Q)\geq 2^{-N+1}}U_Q\right),
\end{equation}
so that $\overline{\Omega_N}$ is the union of fattened Whitney cubes $I^*=(1+\lambda)I$, with $\ell(I)\gtrsim 2^{-N}$, and the boundary of $\Omega_N$ consists of portions of faces of $I^*$ with $\ell(I)\approx 2^{-N}$.
By virtue of Lemma \ref{lemma2.30}, each $\Omega_N$ satisfies the ADR,
Corkscrew and Harnack Chain properties.  We note that, for each of these properties, the constants are uniform in $N$, and depend only on dimension and on the corresponding constants for $\Omega$.

By construction $\Omega_N$ satisfies the qualitative exterior corkscrew condition
(Definition \ref{qualextCS}) since it has exterior corkscrew points at
all scales $\lesssim 2^{-N}$.  By Lemma \ref{lemma:inher-quali} the same statement applies to the Carleson boxes $T_Q$ and $T_\Delta$, and to the sawtooth domains $\Omega_\F$ and $\Omega_{\F,Q}$  (all of them relative to $\Omega_N$)  and even to Carleson boxes within sawtooths.

We write $\omega_N$ for the corresponding harmonic measure and $k_N$ for the corresponding Poisson kernel (we know by \cite{DJe} that $\omega_N$ is absolutely continuous with respect to surface measure, since $\Omega_N$ enjoys a {\it qualitative} 2-sided Corkscrew condition, i.e.,
it has exterior Corkscrew points at all scales $\lesssim 2^{-N}$).
We are going to show that the scale invariant  estimate \eqref{eqn:scale-inva:converse} passes
uniformly to the approximating domains.
To be precise, for all surface balls
$\Delta^N$, defined with respect to the approximating domain $\Omega_N$  (i.e., $\Delta^N=B\cap \pom_N$ with $B$ centered at $\pom_N$), we have
\begin{equation}\label{eqn:scale-inva:converse-approx}
\int_{\Delta^N} \left(k_N^{X_{\Delta^N}}\right)^{\tilde{p}} d\sigma_N \leq \widetilde{C}\, \sigma_N(\Delta^N)^{1-{\tilde{p}}},
\end{equation}
where $\widetilde{C}$ and $1<\tilde{p}\le p$ are independent of $N$.  To prove \eqref{eqn:scale-inva:converse-approx},
we shall first need to establish several preliminary facts.

Given $X\in\Omega_N\subset\Omega$  we write $\hat{x}$ for a point in $\pom$ such that $\delta(X)=|X-\hat{x}|$, analogously $\hat{x}_N\in\pom_N$  with $\delta_N(X)=|X-\hat{x}_N|$ (here $\delta_N$ stands for the distance to the boundary of $\Omega_N$). Set $\Delta_X=B(\hat{x},\delta(X))\cap\pom$ and $\Delta_X^N=B(\hat{x}_N,\delta_N(X))\cap\pom_N$.

Let  $X\in \Omega_N\subset\Omega$ with $\delta_N(X)\approx 2^{-N}$, thus $\delta(X)\approx 2^{-N}$. We claim that for every $Y\in \Omega_N$ we have
\begin{equation}\label{CFMS-approx}
\frac{\omega_N^Y (\Delta_X^N)}{\sigma_N(\Delta_X^N)}
\le
C\,\frac{\omega^Y (\Delta_X)}{\sigma(\Delta_X)},
\end{equation}
where $C$ is independent of $N$.

We first consider the case $\Omega_N$ bounded. For fixed $Y\in \Omega_N$ we set $u(Y)=\omega_N^Y(\Delta_X^N)$ and $v(Y)=\omega^Y(\Delta_X)$. Note that $u$, $v$ are harmonic in $\Omega_N$. For every $Y\in \Delta_X^N$, we have $u(Y)=1\lesssim \omega^Y(\Delta_X)=v(Y)$ by Lemma \ref{Bourgainhm} and the Harnack chain condition. Also, if $Y\in \pom_N\setminus \Delta_X^N$ we have $u(Y)=0\le \omega^Y(\Delta_X)=v(Y)$. Thus, maximum principle yields that $u(Y)\lesssim v(Y)$ for every $Y\in \Omega_N$ and then we obtain as desired
$$
\frac{\omega_N^Y (\Delta_X^N)}{\sigma_N(\Delta_X^N)}
\le
C\,\frac{\omega^Y (\Delta_X)}{\sigma(\Delta_X)},
$$
where $C$ is independent of $N$ (notice that we have used that $\sigma_N(\Delta_X^N)\approx \delta_N(X)^n\approx\delta(X)^n\approx\sigma(\Delta_X)$.)

Let us treat the case $\Omega_N$ unbounded. Let $M\gg 1$ and set $B^N_M=B(\hat{x}_N,M\,\delta_N(X))$ and $\Delta^N_M=B^N_M\cap \pom_N$. Take $\Omega_{N,M}=T_{\Delta^N_M}$ where $T_{\Delta^N_M}\subset \Omega_N$ denotes the Carleson box corresponding to $\Delta^N_M$ for the domain $\Omega_N$. By \eqref{eq2.ball-box} and \eqref{eq2.contain} (with $\Omega$ replaced by $\Omega_N$) we have
$$
\frac54 B_N^M\cap\Omega_N\subset \Omega_{N,M},
\qquad\qquad
\overline{\Omega_{N,M}}\subset \kappa_0 B_N^M \cap\overline{\Omega_N}.
$$
These and the fact that $M\gg 1$ allow us to obtain that $X\in \Omega_{N,M}$, $\hat{x}_N\in \pom_N\cap \pom_{N,M}$, $\delta_{\Omega_{N,M}}(X)=\delta_N(X)\approx  2^{-N}$, and
$$
\Delta_X^N
=
B(\hat{x}_N,\delta_N(X))\cap\pom_N
=
B(\hat{x}_N,\delta_N(X))\cap\pom_{N,M}.
$$
Then, we proceed as in the previous case and consider $u_M(Y)=\omega_{\Omega_{N,M}}^Y(\Delta_X^N)$ and $v(Y)=\omega^Y(\Delta_X)$ for $Y\in \Omega_{N,M}$. The same argument above and the maximum principle in the bounded domain $\Omega_{N,M}$ yields $u_M(Y)\lesssim v(Y)$ for every $Y\in \Omega_{N,M}$ with constant independent on $N$, $M$ and $X$. On the other hand notice that the solutions $u_M(Y)= \omega_{\Omega_{N,M}}^Y(\Delta_X^N)$ are monotone increasing on any fixed $\Omega_{N,M_0}$, as $M_0\leq M\to\infty$, by the maximum principle.  We then obtain that $u_M(Y)\to u(Y):=\hm_N^X(\Delta_X^N)$,
uniformly on compacta, by Harnack's convergence theorem. Thus, $u(Y)\lesssim v(Y)$ for every $Y\in \Omega_N$ and then the previous argument leads to the desired estimate.

\begin{proposition}\label{prop:RHp-scale-N}
There exists $1<\tilde{p}\le p$ such that for all $N\gg 1$ if we let $\Delta^N=B\cap\Omega_N$ be a surface ball with $r(B)\approx 2^{-N}$, then
\begin{equation}\label{RHp-scale-N-cks}
\left(\fint_{\Delta^N} \left(k_N^{X_{\Delta^N}}\right)^{\tilde{p}}\,d\sigma_N\right)^{\frac1{\tilde{p}}}
\lesssim
\sigma_{N}(\Delta^N)^{-1}
\end{equation}
where the constant is uniform in $N$. Furthermore, for every  $Y\in\Omega_N\setminus 2\kappa_0B$ we have
\begin{equation}\label{RHp-scale-N}
\left(\fint_{\Delta^N} \left(k_N^Y\right)^{\tilde{p}}\,d\sigma_N\right)^{\frac1{\tilde{p}}}
\lesssim
\frac{\omega_N^Y(\Delta^N)}{\sigma_N(\Delta^N)},
\end{equation}
where the constant is uniform in $N$.
\end{proposition}

\begin{proof}
Let us momentarily assume \eqref{RHp-scale-N-cks} and we obtain \eqref{RHp-scale-N}.  As observed above $\Omega_N$ is a $1$-sided NTA domain with ADR boundary and satisfies the qualitative exterior corkscrew condition.  Thus we can use Corollary \ref{cor2.poles}: we take $\Delta^N=B\cap\Omega_N$ with $r(B)\approx 2^{-N}$ and
for every  $Y\in\Omega_N\setminus 2\kappa_0B$ we have
$$
\fint_{(\Delta^N)'} k_N^Y\,d\sigma_N
\approx
\omega_N^Y(\Delta^N)\,\fint_{(\Delta^N)'} k_N^{X_{\Delta^N}}\,d\sigma_N,
\qquad B'\subset B.
$$
Then for $\sigma_N$-a.e. $x\in\Delta^N$ we take $B'=B(x,r)$ and let $r\to 0$ to obtain $k_N^Y(x)\approx \omega_N^Y(\Delta^N)\,k_N^{X_{\Delta^N}}(x)$ with constants that do not depend on $N$. Consequently, we have
$$
\left(\fint_{\Delta^N} \left(k_N^Y\right)^{\tilde{p}}\,d\sigma_N\right)^{\frac1{\tilde{p}}}
\lesssim
\omega_N^Y(\Delta^N)
\,
\left(\fint_{\Delta^N} \left(k_N^{X_{\Delta^N}}\right)^{\tilde{p}}\,d\sigma_N\right)^{\frac1{\tilde{p}}},
\qquad
Y\in\Omega_N\setminus 2\kappa_0B.
$$
This and \eqref{RHp-scale-N-cks} clearly give \eqref{RHp-scale-N}.

We next show \eqref{RHp-scale-N-cks}. Write $\Delta^N=\Delta^N(x_N,r)=B(x_N,r)\cap \pom_N$ with $x_N\in \pom_N$ and $r\approx 2^{-N}$. We take $x\in \pom$
such that $\delta(x_N)=|x_N-x|\approx 2^{-N}$ and set
$\Delta=\Delta(x,M\,2^{-N})$ with  $M$ (independent of $N$) large enough.
Note that
\begin{equation}\label{cover}
B(x_N,4r)\cap\Omega_N \subset \cup_I\,I^*\,,
\end{equation}
where the union
runs over a collection of Whitney boxes $I\in\W(\Omega)$,
of uniformly bounded cardinality, all with ${\rm int }(I^*)\subset \Omega_N$, $\ell(I)\approx2^{-N}$,
and $B(x_N,4r)\cap I^*\neq\emptyset$. Notice that we have $B(x_N,4r)\cap\partial\Omega_N\subset \cup_I \partial I^*$.
For each such $I$,
there is a $Q_I\in\dd_{\F_N}$ such that
$$
\ell(I)\approx \ell(Q_I)\approx 2^{-N}\approx\dist(I,Q_I)\,,
$$
so by definition $I\in \W_{Q_I}^*$.
Thus, for any such $I$ we have that
\begin{equation}\label{cover2}
|x_{Q_I}-x|
\lesssim
\ell(Q_I)+d(Q_I,I)+\ell(I)+4r+|x_N-x|
\le
M\,2^{-N}\,,
\end{equation}
for $M$ chosen large enough, and therefore $x_{Q_I}\in \Delta$.

Let $k(\Delta)$ denote the unique $k \in \mathbb{Z}$ such that $2^{-k-1}<  M\,2^{-N}
\leq 2^{-k}$ (below we may need to make $M$ larger). We define
$$
\mathbb{P}:={\rm int}\left(\bigcup_{Q\in \dd^\Delta} \overline{\Omega_{\F_N,Q}}\right).
$$
where $\dd^\Delta$ is defined in \eqref{eq2.box3} (note that $k(\Delta)$ is slightly different.)

Let us make some observations about this set $\mathbb{P}$. Note first that $\dd_{\F_N}$ is the collection of all dyadic cubes whose side length are at least  $2^{-N+1}$. Thus for every $Q\in \dd^\Delta$, we have that $\ell(Q)=2^{-k(\Delta)}$ and $\dd_{\F_N,Q}$ consists of all dyadic subcubes of $Q$ whose side length is at least  $2^{-N+1}$, i.e., those $Q'\in\dd_Q$ with $2^{-N+1}\le \ell(Q)\le 2^{-k(\Delta)}$ and these are a finite number. Note also that  the cardinality of $\dd^\Delta$ is finite and depends only on
$M$ and the ADR constants,
and therefore $\overline{\mathbb{P}}$ is the union of a finite number
(the cardinality depends only on  $M$ and on the parameters that appear in the definition of $\W_Q^*$) of fattened Whitney boxes that are ``close'' to $\Delta$ and have size comparable to $2^{-N}$.
Note also, that $\mathbb{P}\subset \Omega_N$. It is easy to see that  $B(x_N,4r)\cap \Omega_N\subset \mathbb{P}$.
Namely,  for every $Y\in B(x_N,4r)\cap \Omega_N$ we have shown above using \eqref{cover}  that there exists $I^*\ni Y$ with $\ell(I)\approx 2^{-N}$, $I\in \W_{Q_I}^*$ and $x_{Q_I}\in \Delta$. Therefore there is $Q^*\in\dd^\Delta$ such that $Q_I\subset Q^*$. Thus, $Y\in \overline{\Omega_{\F_N,Q_I}}\subset \bigcup_{Q\in \dd^\Delta} \overline{\Omega_{\F_N,Q}}$.
Notice that $B(x_N, 4  r)
\cap \Omega_N$ is open so there exists a ball $B_Y$ centered at $y$ that is contained in this set. For every $Z\in B_Y$ we have just shown that $z\in \bigcup_{Q\in \dd^\Delta} \overline{\Omega_{\F_N,Q}}$, this eventually leads to $Y\in\mathbb{P}$.

\begin{lemma}\label{lemma:Poly-NTA}
$\mathbb{P}$ is an NTA domain  with ADR boundary,   that is, it satisfies the interior {\bf and }exterior corkscrew conditions, the (interior) Harnack chain condition,   and $\partial\mathbb{P}$ is $n$-dimensional ADR;   moreover, all of the corresponding constants  are independent of $N$.
Consequently, $\omega_\mathbb{P}$, the harmonic measure for the domain $\mathbb{P}$, is absolutely continuous with respect to surface measure and its Poisson kernel, $k_\mathbb{P}$, satisfies the scale invariant estimate
\begin{equation}\label{eqn:scale-inva:P}
\int_{\Delta^\mathbb{P}} \left(k_\mathbb{P}^{X_{\Delta^\mathbb{P}}}\right)^{\hat{p}} d\sigma_{\mathbb{P}} \leq C\, \sigma_{\mathbb{P}}(\Delta^{\mathbb{P}})^{1-\hat{p}},
\end{equation}
for some $1<\hat{p}<\infty$ and for all surface balls $\Delta^{\mathbb{P}}$, where $C$ and $\hat{p}$ are independent of $N$.
\end{lemma}

The proof of this result is given below.

Fix $y\in \Delta^N=B(x_N,r)\cap\pom_N$. Then $y\in \partial I^*$ with $I\in\W(\Omega)$,  $\ell(I)\approx 2^{-N}$ and ${\rm int}(I^*)\subset   \Omega_N$. We recall that $\overline{\mathbb{P}}$ is a finite union  of fattened Whitney boxes $J^*\in\W(\Omega)$, and we define $\mathbb{P}_{I^*}$ to be the interior of the union of those boxes $J^*\subset\mathbb{P}$ that overlap $I^*$.
It is easy to see that $I\subset \mathbb{P}_{I^*}\subset \mathbb{P}$.
We shall see from the proof of Lemma \ref{lemma:Poly-NTA} that
$\mathbb{P}_{I^*}$ is an NTA domain
with constants independent of $N$.  Moreover, we claim that taking $M$ above larger we have that $\mathbb{P}_{I^*}$ is {\it strictly} contained in $\mathbb{P}$,
 and in fact $\mathbb{P}\setminus \mathbb{P}_{I^*}$ contains some Whitney box $J\in\W(\Omega)$.  We may verify this claim as follows.
We fix $Q_I\in\F_N$ with $\ell(Q_I)\approx\ell(I)\approx 2^{-N}$ and $\dist(I,Q_I)\approx 2^{-N}$.  Notice that, for $M$ large enough, there is some $Q\in \dd^{\Delta}$ such that
$Q_I\subset Q$
(cf. \eqref{cover}-\eqref{cover2}.)
For this $Q$,
take $J\in \W_{Q}^*$ which is clearly in $\mathbb{P}$. Note that
$$
\ell(J)\approx\ell(Q)\approx M\,2^{-N}\approx M\,\ell(I)
$$
and therefore for $M$ large enough we have that $I^*\cap J^*=\emptyset$,
by properties of Whitney cubes (since
otherwise $\ell(I)\approx \ell(J)$). This shows that $\mathbb{P}_{I^*}\subsetneq \mathbb{P}$.

Set $u(Z)=G_N(Z,X_{\Delta^N})$ and $v(Z)=G_{\mathbb{P}}(Z,X_{\Delta^N})$ where $G_N$, $G_{\mathbb{P}}$ are respectively the Green functions associated with the domains $\Omega_N$ and $\mathbb{P}$. Notice that $u$ and $v$ are non-negative harmonic functions in $\mathbb{P}_{I^*}$ (since $\mathbb{P}_{I^*}\subset \mathbb{P}\subset \Omega_N$) and we can assume that $X_{\Delta^N}\notin \mathbb{P}_{I^*}$:  indeed, by the Harnack chain condition,
and our observation above that  $\mathbb{P}_{I^*}\subsetneq \mathbb{P}$,
we may assume that $X_{\Delta^N}$ is the center of a Whitney box  $J\subset
\mathbb{P}\setminus \mathbb{P}_{I^*}$.
Notice that $u\big|_{2\,\Delta^N}=v\big|_{2\,\Delta^N}=0$
(since $2\,\Delta^N\subset\pom_N$ and since $B(x_N,4\,r)\cap\Omega_N\subset \mathbb{P}$ implies that $2\,\Delta^N=B(x_N,2\,r)\cap\partial \mathbb{P}$). Therefore $u=v\equiv 0$ in $B(y,2^{-N-N_1})\cap\partial\mathbb{P}_{I^*}$ (here $N_1$ is a fixed big integer such that $B(y,2^{-N-N_1})\cap\partial\mathbb{P}_{I^*}\subset 2\Delta^N$. As mentioned before, $\mathbb{P}_{I^*}$ is an NTA
domain with constants independent of $N$ and we have the (full) comparison principle (see \cite[Lemma 4.10]{JK}):
$$
\frac{u(Z)}{v(Z)}\approx \frac{u(X(I))}{v(X(I))},
$$
for every $Z\in B(y, 2^{-N-N_1}/C)\cap \mathbb{P}_{I^*}$ with $C$ large independent of $N$. We remind the reader that $X(I)$ stands for the center of the Whitney box $I$. Notice that
$$|X(I)-X_{\Delta^N}|\approx \dist(X(I),\partial\mathbb{P})\approx
\dist(X_{\Delta^N},\partial\mathbb{P})\approx 2^{-N}\,.$$
Thus, the Harnack chain condition and the size estimates for the Green functions imply that $u(X(I))/v(X(I))\approx 1$. Hence $u(Z)\approx v(Z)$ for all $Z\in B(y, 2^{-N-N_1}/C)\cap \mathbb{P}_{I^*}$.


As observed above $\Omega_N$ is a 1-sided NTA with ADR boundary and also satisfies the qualitative exterior corkscrew condition, so that, in particular, we may apply Lemma \ref{lemma2.cfms}.
We take   $s\ll 2^{-N}$ and write $\Delta^N(y,s)=B(y,s)\cap \Omega_N$ to obtain
$$
\frac{u(X_{\Delta^N(y,s)})}{s}=\frac{G_N(X_{\Delta^N(y,s)},X_{\Delta^N})}{s}\approx \frac{\omega_N^{X_{\Delta^N}}(\Delta^N(y,s))}{\sigma_N(\Delta^N(y,s))}.
$$
The same can be done with $G_{\mathbb{P}}$ (indeed $\mathbb{P}$ is NTA) after observing that  $\Delta^N(y,s)$ is also a surface ball for $\mathbb{P}$ and that after using Harnack if needed $X_{\Delta^N(y,s)}$ is a corkscrew point with respect to $\mathbb{P}$ for that surface ball. Then,
$$
\frac{v(X_{\Delta^N(y,s)})}{s}=\frac{G_\mathbb{P}(X_{\Delta^N(y,s)},X_{\Delta^N})}{s}\approx \frac{\omega_\mathbb{P}^{X_{\Delta^N}}(\Delta^N(y,s))}{\sigma_\mathbb{P}(\Delta^N(y,s))}
$$
Note that if $s$ is small enough then $Z=X_{\Delta^N(y,s)}\in B(y, 2^{-N-N_1}/C)\cap \mathbb{P}_{I^*}$. This allows us to gather the previous estimates and conclude that
$$
\frac{\omega_N^{X_{\Delta^N}}(\Delta^N(y,s))}{\sigma_N(\Delta^N(y,s))}
\approx
\frac{\omega_\mathbb{P}^{X_{\Delta^N}}(\Delta^N(y,s))}{\sigma_\mathbb{P}(\Delta^N(y,s))}.
$$
Next we use the Lebesgue differentiation theorem and obtain that $k_N^{X_{\Delta^N}}(y)\approx k_\mathbb{P}^{X_{\Delta^N}}(y)$ for a.e. $y\in \Delta^N$.  Then if we set $\tilde{p}=\min\{p,\hat{p}\}$ and use \eqref{eqn:scale-inva:P} we obtain
\begin{multline*}
\left(\fint_{\Delta^N} \left(k_N^{X_{\Delta^N}}\right)^{\tilde{p}}\,d\sigma_N\right)^{\frac1{\tilde{p}}}
\lesssim
\left(\fint_{\Delta^N} \left(k_\mathbb{P}^{X_{\Delta^N}}\right)^{\tilde{p}}\,d\sigma_{\mathbb{P}}\right)^{\frac1{\tilde{p}}}
\\
\le
\left(\fint_{\Delta^N} \left(k_\mathbb{P}^{X_{\Delta^N}}\right)^{\hat{p}}\,d\sigma_{\mathbb{P}}\right)^{\frac1{\hat{p}}}
\lesssim
\sigma_{\mathbb{P}}(\Delta^N)^{-1}
\approx
\sigma_{N}(\Delta^N)^{-1}.
\end{multline*}
This completes the proof of Proposition \ref{prop:RHp-scale-N}.
\end{proof}

\begin{proof}[Proof of Lemma \ref{lemma:Poly-NTA}]
The proof is very elementary. The corkscrew conditions are as follows. Fix $\Delta_{\mathbb{P}}=\Delta_{\mathbb{P}}(x,r)$ with $0<r<\diam(\partial \mathbb{P})\approx 2^{-N}$. Since $x\in \partial\mathbb{P}$, which is a finite union of fattened boxes of size comparable to $2^{-N}$, $x\in \partial I^*$ with $\ell(I)\approx 2^{-N}$. The fact that $I^*$ is a cube yields easily that you can always find a corkscrew point in the segment that joints $x$ and $X(I)$ (the center of $I$) with constants independent of $N$ and this is an interior corkscrew point for $\mathbb{P}$. For the exterior corkscrew condition we notice that $\partial I^*$ can be covered by the union of the Whitney boxes that are neighbors of $I$ (i.e., whose boundaries touch). Then $x\in J$ with $\partial I\cap\partial J\neq\emptyset$. Note that $J\subset{\rm int}(J^*)$ and hence $J^*$ cannot be one of the Whitney boxes that define $\mathbb{P}$. Take $J'$ a neighbor of $J$: if ${\rm int}((J')^*)\subset \mathbb{P}$, then $(J')^*$ ``bites'' a small portion of  $J$ (since $(J')^*$ is a small dilation of $J'$); otherwise $(J')^*$ does not ``bite'' $J$. Eventually, we see that the part of
$J$ minus all these ``bites'' is contained in $\mathbb{P}^c$. Note that $x$ is in one of this portions and therefore we can find a corkscrew point in the segment between $x$ and the center of $J$. This is possible since $r\lesssim 2^{-N}$.

We show the Harnack chain condition. Let $X_1, X_2\in\mathbb{P}$. Then, for each $i=1,2$ there exists $Q_i\in\dd^{\Delta}$ and $Q_i'\in\dd_{\F_N,Q_i}$ with $X_i\in {\rm int}(I_i^*)$ where $I_i\in \W^*_{Q_i'}$. Note that $\ell(Q_i)\approx \ell(Q_i')\approx \ell(I_i)\approx 2^{-N}$. If $I_1^*$, $I_2^*$ overlap then the condition is clear. Otherwise, $\dist(I_1^*,I_2^*)\approx\dist(I_1, I_2)\approx 2^{-N}$ and $|X_1-X_2|\approx 2^{-N}$. Moreover, we can construct an appropriate chain in each fattened box between
$X_i$ and $X(I_i)$.  In fact, if $\delta_{\mathbb{P}}(X_i)\approx 2^{-N}$
the number of balls is bounded by
a dimensional constant, whereas
if $\delta_{\mathbb{P}}(X_i)\ll 2^{-N}$ then we simply use that each cube $I_i^*$ has the Harnack chain property with uniform constants. Therefore, it suffices to assume that $X_i$  is the center of $I_i^*$. In such a case $\delta_{\mathbb{P}}(X_i)\approx 2^{-N}$, $i=1,2$, $|X_1-X_2|\approx 2^{-N}$ and therefore we need to find a Harnack chain whose cardinality is uniformly bounded (in $N$).

Therefore we fix two fattened Whitney boxes $I_1^*$, $I_2^*$ and we want to find a Harnack chain between the centers of such boxes. As just mentioned, if $I_1^*$, $I_2^*$ overlap we can easily construct a finite (depending only on $\lambda$) chain between $X(I_1)$ and $X(I_2)$ (with respect to the domain $I_1^*\cup I_2^*$) such that $\diam (B_i)\approx 2^{-N}$, $\dist(B_i, \partial (I_1^*\cup I_2^*))\approx 2^{-N}$. Let us observe that
in such a case
$$
\dist(B_i,\partial\mathbb{P})\le \diam(\mathbb{P})
\approx 2^{-N}
\approx
\dist(B_i, \partial (I_1^*\cup I_2^*))
\le
\dist(B_i,\partial\mathbb{P})
$$
Thus the constructed Harnack chain is valid for the domain $\mathbb{P}$. If  $I_1^*$, $I_2^*$ do not overlap it suffices to find a chain of fattened boxes $J_1^*,\dots,J_K^*$ in $\mathbb{P}$ such that $J_1^*=I_1^*$, $J_K^*=I_2^*$ and with the property that $J_k^*$ and $J_{k+1}^*$ overlap. Notice that
$K$ is uniformly bounded  since $\mathbb{P}$ is a finite union of fattened Whitney boxes with
cardinality bounded uniformly in $N$.
Notice that if $Q_1$, $Q_2\in \dd^{\Delta}$
we have $\ell(Q_1)=\ell(Q_2)$ and $\dist(Q_1,Q_2)\le 4\,\ell(Q_2)$.
Therefore, the construction of the sets $\W_{Q}$ guarantees that
$\W_{Q_1}\cap \W_{Q_2}$ is non-empty (cf. \eqref{eq2.whitney2**})
and there exists $I\in \W_{Q_1}\cap \W_{Q_2}$.
This leads to show that $I\subset {\rm int}(U_{Q_1})\cap {\rm int}(U_{Q_2})$ which in turn gives $I\subset \Omega_{\F_N,Q_1}\cap \Omega_{\F_N,Q_2}$. To conclude we just observe that each $\Omega_{\F_N,Q_i}$ enjoys the Harnack chain property  by Lemma \ref{lemma2.30}  and this means that we can connect any box in $\Omega_{\F_N,Q_i}$ with $I$ as before.

To show that $\partial\mathbb{P}$ is ADR we proceed as follows. Notice that $\partial\mathbb{P}\subset\cup_{Q\in\dd^{\Delta}} \partial \Omega_{\F_N,Q}$ where each $\partial \Omega_{\F_N,Q}$ is ADR by Lemma \ref{lemma2.30} and the cardinality of $\dd^\Delta$ is finite and depends only on $M$ and the ADR constants of $\Omega$. Thus the upper bound of the ADR condition for $\partial\mathbb{P}$ follows at once with constants that are independent of $N$. For the lower bound, let $B=B(x,r)$ with $x\in \partial\mathbb{P}$ and $r\le\diam \mathbb{P}\approx 2^{-N}$. As $\partial\mathbb{P}$ is comprised of partial faces of fattened Whitney cubes $I^*$ with $\ell(I)\approx 2^{-N}$, then $x$ lies in a subset $F$ of a (closed)
face of $I^*$,  with $F\subset B\cap\partial\mathbb{P}$, and $H^n(F\cap B)=H^n(F)\gtrsim r^n$,
as desired. Again the constants are independent of $N$.

Having shown that $\mathbb{P}$ is an NTA domain with ADR boundary,   with constants independent of $N$, we may invoke \cite{DJe} to deduce absolute continuity of harmonic measure,
along with the desired scale invariant estimate \eqref{eqn:scale-inva:P} for the Poisson kernel, with constants independent of $N$.
\end{proof}

With these preliminaries in hand, we now turn to the proof of
\eqref{eqn:scale-inva:converse-approx}.  To this end, we fix $\Delta^N$, and
consider three cases. If $r(\Delta^N)\approx 2^{-N}$ then
\eqref{RHp-scale-N-cks}
gives the desired estimate.  Assume now that  $r(\Delta^N)\gg 2^{-N}$. We cover $\Delta^N$ by those dyadic cubes $Q^N\in \mathbb{D}_N(\pom_N)$ that meet $\Delta^N$ (we remind the reader that $\mathbb{D}_N(\pom_N)$ are the dyadic cubes associated to the ADR set $\pom_N$ with $\ell(Q_N)=2^{-N}$). Note that for each of those cubes we have $X_{\Delta^N}\in\Omega_N\setminus 2\kappa_0 B(x_{Q^N}, C\,2^{-N})$ and then we can use \eqref{RHp-scale-N} with $Y=X_{\Delta^N}$ and the surface ball $\Delta^N(x_{Q^N}, C\,2^{-N})$:
\begin{align*}
\int_{\Delta^N} \left(k_N^{X_{\Delta^N}}\right)^{\tilde{p}}\,d\sigma_N
&\le
\sum_{Q^N} \sigma_N(Q^N)\,\fint_{Q^N} \left(k_N^{X_{\Delta^N}}\right)^{\tilde{p}}\,d\sigma_N
\\
&\lesssim
\sum_{Q^N} \sigma_N(Q^N)\,\fint_{\Delta^N(x_{Q^N}, C\,2^{-N})} \left(k_N^{X_{\Delta^N}}\right)^{\tilde{p}}\,d\sigma_N
\\
&\lesssim
\sum_{Q^N} \sigma_N(Q^N)\,\left(\frac{\omega_N^{X_{\Delta^N}}\big(\Delta^N(x_{Q^N}, C\,2^{-N})\big)}{\sigma_N\big(\Delta^N(x_{Q^N}, C\,2^{-N})\big  )}
\right)^{\tilde{p}}
\\
&\lesssim
\sum_{Q^N} \sigma_N(Q^N)\,\left(\frac{\omega_N^{X_{\Delta^N}}(\Delta^N_{X_{Q^N}} )}{\sigma_N(\Delta^N_{X_{Q^N}})}\right)^{\tilde{p}},
\end{align*}
where in the last estimate we have used that $\omega_N$ is doubling: let us observe that $\Delta^N(x_{Q^N}, C\,2^{-N})$ and $\Delta^N_{X_{Q^N}}$ are balls whose radii are comparable to $2^{-N}$ and the distance of their centers is controlled by $2^{-N}$, and that $\delta_N(X_{\Delta^N})\approx r(\Delta^N)\gg 2^{-N}$.  Let us observe that $X_{Q^N}\in \Omega_N$ and
$\delta(X_{Q^N})\approx \delta_N(X_{Q^N})\approx 2^{-N}\ll\delta_N(X_{\Delta^N})\approx \delta(X_{\Delta^N})$. Then we use \eqref{CFMS-approx} to obtain
\begin{equation}\label{RHp-discretize-N}
\int_{\Delta^N} \left(k_N^{X_{\Delta^N}}\right)^{\tilde{p}}\,d\sigma_N
\lesssim
\sum_{Q^N} \sigma_N(Q^N)\,\left(\frac{\omega^{X_{\Delta^N}}(\Delta_{X_{Q^N}} )}{\sigma(\Delta_{X_{Q^N}})}\right)^{\tilde{p}}
\end{equation}

Given $Q^N$, its center $x_{Q^N}$ is in $\pom_N\subset\Omega$ and therefore there exist $\tilde{x}_{Q^N}\in\pom$ such that $\delta(x_{Q^N})=|x_{Q^N}-\tilde{x}_{Q^N}|\approx 2^{-N}$. Then using the dyadic cube structure in $\pom$, there exists a unique $R(Q^N)\in \mathbb{D}_N(\pom)$ that contains $\tilde{x}_{Q^N}$. We note that $\Delta_{X_{Q^N}} $ and $\Delta_{R(Q^N)}$ are surface balls in $\pom$ with radii comparable to $2^{-N}$ and whose centers are separated at most $C\,2^{-N}$. This means that $\Delta_{X_{Q^N}} \subset C\,\Delta_{R(Q^N)}$. On the other hand, one can show that $C\,\Delta_{R(Q^N)}\subset C'\,\Delta_{X_{\Delta^N}}$ by using that $Q^N$ meets $\Delta^N$ and that $r(\Delta^N)\gg 2^{-N}$.

Next we want to control the overlapping of the family $\{C\,\Delta_{R(Q^N)}\}_{Q^N}$. It is easy to see that if $C\,\Delta_{R(Q^N) }\cap C\,\Delta_{R(Q_1^N)}\neq\emptyset$
then $|x_{Q^N}-x_{Q_1^N}|\lesssim 2^{-N}=\ell(Q_1^{N})=\ell(Q^{N})$. Thus the fact that the cubes $Q^N$ are all dyadic disjoint with side length $2^{-N}$ yields
$$
\#\{Q^N: C\,\Delta_{R(Q^N)}\cap C\,\Delta_{R(Q_1^N)}\neq\emptyset \}
\le
\#\{Q^N: |x_{Q^{N}}-x_{Q_1^{N}}|\lesssim 2^{-N}\}\le C.
$$

Gathering the previous facts we obtain that
$$
\sum_{Q^N} \chi_{\Delta_{X_{Q^N}}}
\le \sum_{Q^N } \chi_{C\,\Delta_{R(Q^N)}}
\le C\,\chi_{C'\,\Delta_{X_{\Delta^N}}}.
$$
Set $\Delta_0=C'\,\Delta_{X_{\Delta^N}}$ and observe that by the Harnack Chain property $\omega^{X_{\Delta^N}}\approx \omega^{X_{\Delta_0}}$. Thus
plugging these estimates into \eqref{RHp-discretize-N} we conclude that
\begin{multline*}
\int_{\Delta^N} \left(k_N^{X_{\Delta^N}}\right)^{\tilde{p}}\,d\sigma_N
\lesssim
\sum_{Q^N} \sigma_N(Q^N)\,\left(\frac{\omega^{X_{\Delta_0}}(\Delta_{X_{Q^N}} )}{\sigma(\Delta_{X_{Q^N}})}\right)^{\tilde{p}}
\lesssim
\sum_{Q^N} \int_{\Delta_{X_{Q^N}}} \left(k^{X_{\Delta_0}}\right)^{\tilde{p}}\,d\sigma
\\
\lesssim
\int_{\Delta_0}\left(k^{X_{\Delta_0}}\right)^{\tilde{p}}\,d\sigma
\le
\left(\int_{\Delta_0}\left(k^{X_{\Delta_0}}\right)^{p}\,d\sigma\right)^{\frac{\tilde{p}}{p}}\,\sigma(\Delta_0)^{\frac{p-\tilde{p}}{p}}
\lesssim
\sigma(\Delta_0)^{1-\tilde{p}}
\approx
\sigma_N(\Delta^N)^{1-\tilde{p}},
\end{multline*}
where we have used $\tilde{p}\le p$ and our hypothesis \eqref{eqn:scale-inva:converse} and where the involved constants are independent of $N$. This completes the case $r(\Delta^N)\gg 2^{-N}$.

To complete Step 1, we need to
establish \eqref{eqn:scale-inva:converse-approx}
for $r:=r(\Delta^N)\ll 2^{-N}$.  Set $\Delta^N=B(x_N,r)\cap\pom_N$.
Given $\Delta^N(x_N,C\,2^{-N})$ we consider the
associated $\mathbb{P}$ as before. For a fixed $y\in \Delta^N=B(x_N,r)\cap\pom_N$,
we have that $y\in \partial I^*$ with $\ell(I)\approx 2^{-N}$ and ${\rm int}(I^*)\subset \Omega_N$.   We define $\mathbb{P}_{I^*}$ as before
and recall that it is an NTA domain with constants independent of $N$.
 Let us note that
$5\Delta^N\subset \pom_N\cap \partial\mathbb{P}_{I^*}$,
since $r\ll2^{-N}$.
We let $Y_{\Delta^N}$ denote a Corkscrew point
with respect to $\Delta^N$, for the domain $\mathbb{P}_{I^*}$.
Let $c<1$ be small enough such that
$Y_{\Delta^N}\notin T_{\Delta^N(y,2\,c\,r)}^{\mathbb{P}_{I^*}}$,
where $T_{\Delta^N(y,2\,c\,r)}^{\mathbb{P}_{I^*}}$ is the Carleson box
with respect to $\Delta^N(y,2\,c\,r)$ relative to the domain
$\mathbb{P}_{I^*}$ (note that $c$ depends on the various geometric constants of $\mathbb{P}_{I^*}$ which are all independent of $N$, see \eqref{eq2.contain}). Notice that by Lemma \ref{lemma2.30}, $T_{\Delta^N(y,2\,c\,r)}^{\mathbb{P}_{I^*}}$ inherits the (interior) Corkscrew and Harnack Chain conditions
and also the ADR property
from  $\mathbb{P}_{I^*}$.
 Moreover, $T_{\Delta^N(y,2\,c\,r)}^{\mathbb{P}_{I^*}}$ also inherits
the exterior corkscrew condition, with uniform bounds.  Indeed, consider
a surface ball $\Delta_\star(z,\rho)\subset \partial T_{\Delta^N(y,2\,c\,r)}^{\mathbb{P}_{I^*}},$ with
$z\in \partial T_{\Delta^N(y,2\,c\,r)}^{\mathbb{P}_{I^*}}$, and $\rho\lesssim r
\approx \diam(T_{\Delta^N(y,2\,c\,r)}^{\mathbb{P}_{I^*}})$.   If $\dist(z,\partial\mathbb{P}_{I^*})
\leq \rho/100$, then $B(z,\rho)\cap\big(\ree\setminus T_{\Delta^N(y,2\,c\,r)}^{\mathbb{P}_{I^*}}\big)$ contains an exterior Corkscrew point inherited from $\mathbb{P}_{I^*}$.  On the other hand, if
$\dist(z,\partial\mathbb{P}_{I^*})
> \rho/100$, then $z$ lies on a face of some fattened
Whitney cube $J^*\in\W(\mathbb{P}_{I^*})$, of side length $\ell(J^*)\gtrsim \rho$, for which there is
some adjoining Whitney cube $J'$ containing an exterior Corkscrew point in
$B(z,\rho)\cap \big(\mathbb{P}_{I^*}\setminus T_{\Delta^N(y,2\,c\,r)}^{\mathbb{P}_{I^*}}\big)$.
Consequently $T_{\Delta^N(y,2\,c\,r)}^{\mathbb{P}_{I^*}}$ is an NTA domain with constants independent of $N$, $r$ and $y$. 

Set $u(Z)=G_N(Z,Y_{\Delta^N})$ and $v(Z)=G_{\mathbb{P}_{I^*}}(Z,Y_{\Delta^N})$ where $G_N$, $G_{\mathbb{P}_{I^*}}$ are respectively the Green functions associated with the domains $\Omega_N$ and $\mathbb{P}_{I^*}$. Then $u$ and $v$ are non-negative harmonic functions in $T_{\Delta^N(y,2\,c\,r)}^{\mathbb{P}_{I^*}}$ (since $Y_{\Delta^N}\notin T_{\Delta^N(y,2\,c\,r)}^{\mathbb{P}_{I^*}}$).
Notice that
for $c$ small, $r\ll2^{-N}$, and $y \in \Delta^N=\Delta^N(x_N,r)$ we have that
$$\Delta^N_1:= \Delta^N(y,cr):= \big(B(y,c\,r)\cap\pom_N\big)
\subset\left(\pom_N\cap\partial\mathbb{P}_{I^*}\cap
\partial T_{\Delta^N(y,2\,c\,r)}^{\mathbb{P}_{I^*}}\right)\,.$$
In particular, $u\big|_{\Delta^N_1}=v\big|_{\Delta^N_1}=0$.
As noted above, $T_{\Delta^N(y,2\,c\,r)}^{\mathbb{P}_{I^*}}$ is an NTA domain with constants independent of $N$ and thus by \cite[Lemma 4.10]{JK})
we have the comparison principle:
$$
\frac{u(Z)}{v(Z)}\approx \frac{u(X_0)}{v(X_0)}\,,\qquad \forall
Z\in B(y, cr/C)\cap T_{\Delta^N(y,2\,c\,r)}^{\mathbb{P}_{I^*}}\,,
$$
with $C$ sufficiently large but independent of $N$, and where $X_0\in
T_{\Delta^N(y,2\,c\,r)}^{\mathbb{P}_{I^*}}$ is a corkscrew point associated to the surface ball $\Delta_1^N$, relative to each of the domains $\Omega_N$, $\mathbb{P}_{I^*}$ and
$T_{\Delta^N(y,2\,c\,r)}^{\mathbb{P}_{I^*}}$.   By the Harnack chain condition in $\mathbb{P}_{I^*}$
and since
$$|X_0-Y_{\Delta^N}|\approx \dist(X_0,\partial\mathbb{P}_{I^*})
\approx \dist(Y_{\Delta^N},\partial\mathbb{P}_{I^*}) \approx r \,, $$
one can show that $u(X_0)/v(X_0)\approx 1$ which eventually  leads to $u(Z)\approx v(Z)$ for all $Z\in B(y, cr/C)\cap T_{\Delta^N(y,2\,c\,r)}^{\mathbb{P}_{I^*}}$.

We can now apply Lemma \ref{lemma2.cfms} (since, as observed above, $\Omega_N$ satisfies the required ``qualitative assumption''),
so that if we take  $s\ll r$ and write $\Delta^N(y,s)=B(y,s)\cap \Omega_N$, we have
$$
\frac{u(X_{\Delta^N(y,s)})}{s}=\frac{G_N(X_{\Delta^N(y,s)},X_{\Delta^N})}{s}\approx \frac{\omega_N^{X_{\Delta^N}}(\Delta^N(y,s))}{\sigma_N(\Delta^N(y,s))}.
$$
The same can be done with $G_{\mathbb{P}_{I^*}}$ (indeed $\mathbb{P}_{I^*}$ is NTA) after observing that  $\Delta^N(y,s)$ is also a surface ball for $\mathbb{P}_{I^*}$ and that, after using Harnack if needed, $X_{\Delta^N(y,s)}$ may be taken to be a corkscrew point
with respect to $\mathbb{P}_{I^*}$ for that surface ball. Then,
$$
\frac{v(X_{\Delta^N(y,s)})}{s}=\frac{G_{\mathbb{P}_{I^*}}(X_{\Delta^N(y,s)},X_{\Delta^N})}{s}\approx \frac{\omega_{\mathbb{P}_{I^*}}^{X_{\Delta^N}}(\Delta^N(y,s))}{\sigma_{\mathbb{P}_{I^*}}(\Delta^N(y,s))}
$$
Note that if $s$ is small enough then $Z=X_{\Delta^N(y,s)}\in B(y, cr/C)\cap T_{\Delta^N(y,2\,c\,r)}^{\mathbb{P}_{I^*}}$. This allows us to gather the previous estimates and conclude that
$$
\frac{\omega_N^{X_{\Delta^N}}(\Delta^N(y,s))}{\sigma_N(\Delta^N(y,s))}
\approx
\frac{\omega_{\mathbb{P}_{I^*}}^{X_{\Delta^N}}(\Delta^N(y,s))}{\sigma_{\mathbb{P}_{I^*}}(\Delta^N(y,s))}.
$$
Next we use the Lebesgue differentiation theorem and obtain that $k_N^{X_{\Delta^N}}(y)\approx k_{\mathbb{P}_{I^*}}^{X_{\Delta^N}}(y)$ for a.e. $y\in \Delta^N$. Then
\begin{multline*}
\left(\fint_{\Delta^N} \left(k_N^{X_{\Delta^N}}\right)^{\tilde{p}}\,d\sigma_N\right)^{\frac1{\tilde{p}}}
\lesssim
\left(\fint_{\Delta^N} \left(k_{\mathbb{P}_{I^*}}^{X_{\Delta^N}}\right)^{\tilde{p}}\,d\sigma_{\mathbb{P}_{I^*}}\right)^{\frac1{\tilde{p}}}
\\
\le
\left(\fint_{\Delta^N} \left(k_{\mathbb{P}_{I^*}}^{X_{\Delta^N}}\right)^{\hat{p}}\,d\sigma_{\mathbb{P}_{I^*}}\right)^{\frac1{\hat{p}}}
\lesssim
\sigma_{\mathbb{P}_{I^*}}(\Delta^N)^{-1}
\approx
\sigma_{N}(\Delta^N)^{-1}
\end{multline*}
where we have used that $\tilde{p}\le \hat{p}$ and \eqref{eqn:scale-inva:P}. This completes the proof of \eqref{eqn:scale-inva:converse-approx}

 \noindent{\it Remark.}  We notice that to obtain \eqref{eqn:scale-inva:P}, in place of using \cite{DJe}, we could have invoked \cite{VV}: $\mathbb{P}$ is a polyhedral domain
 (its boundary consists of a finite number of flat ``faces''), is  NTA and has the ADR property, thus $k_{\mathbb{P}}$ is a $RH_2$ weight. Thus $\hat{p}=2$ and  $\tilde{p}=\min\{p,2\}$.

\subsection{Step 2: Local $Tb$ theorem for square functions}\label{section-step2-Tb}

Having already established Step 1, we want to show that $\Omega_N$ has the UR property with uniform bounds. In the last step we shall  show that this ultimately implies that $\Omega$ inherits this property.
Therefore, in all the remaining steps, but the last one, we drop the already fixed subindex/superindex $N$ everywhere and write $\Omega$ to denote the corresponding approximating domain $\Omega_N$. Our main assumption is that \eqref{eqn:scale-inva:converse-approx} holds, and this rewrites as
\begin{equation}\label{eqn:scale-inva:converse-approx:wo-N}
\int_{\Delta} \left(k^{X_{\Delta}}\right)^{\tilde{p}} d\sigma \leq C\, \sigma(\Delta)^{1-{\tilde{p}}},
\end{equation}
where $\tilde{p}$ and $C$ are independent of $N$.

We introduce some notation. We have already defined $\W=\W(\Omega)$ the collection of Whitney boxes of $\Omega$. We also define $\W^{ext}=\W(\Omega_{ext})$ the collection of Whitney boxes of $\Omega_{ext}=\re^{n+1}\setminus\overline{\Omega}$. Set $\W_0=\W\cup\W^{ext}$ ---notice that we can get directly this collection by taking the Whitney decomposition of $\re^{n+1}\setminus\pom$. We can then consider as before $\W_Q$ with the same fixed $C_0$
as before (this guarantees among other things that $\W_Q\neq\emptyset$). With that fixed
$C_0$  we define analogously $\W_Q^{ext}$ (using only $I$'s in $\W^{ext}$). Notice that we have not assumed exterior corkscrew points for $\Omega_{ext}$ and therefore $\W_Q^{ext}$ might be the null set. We then define
$$
V_Q
=
\bigcup_{I\in \W_Q} I,
\qquad
\Lambda(x)
=
\bigcup_{x\in Q\in\dd} V_Q
\qquad
\Lambda_{Q_0}(x)
=
\bigcup_{x\in Q\in\dd_{Q_0}} V_Q
$$
and analogously
$$
V_Q^{ext}
=
\bigcup_{I\in \W_Q^{ext}} I,
\qquad
\Lambda^{ext}(x)
=
\bigcup_{x\in Q\in\dd} V_Q^{ext}
\qquad
\Lambda^{ext}_{Q_0}(x)
=
\bigcup_{x\in Q\in\dd_{Q_0}} V_Q^{ext},
$$
where $x\in\pom$ and $Q_0\in\dd$.
We also consider the ``two-sided''
cones
$\tilde{\Lambda}(x)=\Lambda(x)\cup \Lambda^{ext}(x)$ and $\tilde{\Lambda}_{Q_0}(x)=\Lambda_{Q_0}(x)\cup \Lambda_{Q_0}^{ext}(x)$.

In this step we are going to apply the following
local $Tb$ theorem for square functions in \cite{GM}:

\begin{theorem}[Local $Tb$ theorem for square functions, \cite{GM}]\label{theor:Tb}
Let $\Omega\subset \re^{n+1}$, $n\ge 2$, be a connected open set whose boundary $\partial\Omega$ is ADR.
We assume that there is an exponent $q\in (1,2]$, and a finite constant $\A_0>1$ such that for every $Q\in\dd$ there exists a function $b_Q$ satisfying
\begin{equation}\label{Tb-bQ-Lp}
\int_{\pom} |b_Q|^q\,d\sigma
\le
\A_0\,\sigma(Q)
\end{equation}
\begin{equation}\label{Tb-bQ-below}
\left|\int_Q b_Q\,d\sigma\right|\ge
\frac1{\A_0}\,\sigma(Q)
\end{equation}
\begin{equation}\label{Tb-S-bQ}
\int_{Q}\left(\iint_{\tilde{\Lambda}_{Q}(x)} |\nabla^2 \mathcal{S}b_Q(Y)|^2 \,\frac{dY}{\delta(Y)^{n-1}}
\right)^{\frac{q}2}\,d\sigma(x)
\le
\A_0\,\sigma(Q)
\end{equation}
Then,
\begin{equation}\label{Tb-conclusion}
\iint_{\re^{n+1}} |\nabla^2 \mathcal{S}f(Y)|^2 \,\delta(Y)\,dY
\le
C\,\|f\|_{L^2(\pom)}^2
\end{equation}
\end{theorem}

We observe that we can easily show
$$
\int_{\re^{n+1}} |\nabla^2 \mathcal{S}f(Y)|^2 \,\delta(Y)\,dY
\approx
\int_{\pom}\int_{\tilde{\Lambda}(x)} |\nabla^2 \mathcal{S}f(Y)|^2\,\frac{dY}{\delta(Y)^{n-1}}\,d\sigma(x)
$$
which is nothing but the comparability of the ``vertical'' and ``conical'' square functions. In this way, we see \eqref{Tb-S-bQ} as an $L^q$-testing condition  for the local (conical) square function and the conclusion states the $L^2$ boundedness of the (conical or vertical) square function.

In order to apply this result and define the functions $b_Q$ we shall require some geometric preliminaries. Given $Q\in\dd$, we recall that there exists a surface ball $\Delta_Q=\Delta(x_Q,r_Q)$ such that $\Delta_Q\subset Q\subset C\,\Delta_Q=B(x_Q,C\,r_Q)\cap\pom$ with $r_Q\approx\ell(Q)$.
Let $\kappa_1$ be large enough such that $C\,r_Q<\kappa_1\ell(Q)$ and also with $\kappa_1>\kappa_0$ where $\kappa_0$ is given in \eqref{condition:kappa0} below. Set
$\tilde{B}_Q=B(x_Q,\kappa_1\,\ell(Q))$.
We notice that if $x\in Q$ and $Y\in \Lambda_Q(x)$ then $Y\in V_{Q'}$
with $x\in Q'\in\dd_{Q}$ and $Y\in I\in \W_{Q'}$. Thus,
$$
|Y-x_{Q}|
\le
\ell(I)+d(I,Q')+\ell(Q')+|x-x_Q|
\lesssim
C_0 \ell(Q')+\ell(Q)\lesssim
C_0 \ell(Q).
$$
The same can be done for $Y\in \Lambda^{ext}_Q(x)$ and therefore by taking $\kappa_1$
sufficiently large  (depending on $K_0)$ we have
\begin{equation}\label{local-cone-ball}
\bigcup_{x\in Q} \tilde{\Lambda}(x)
\subset
\tilde{B}_Q.
\end{equation}
Next we set $\hat{B}_Q=\kappa_2\,\tilde{B}_Q$, with $\kappa_2$ large enough so that $\hat{X}_Q$, the corkscrew point relative to $\hat{\Delta}_Q=\hat{B}_Q\cap\Omega$, satisfies $\hat{X}_Q\notin 6\,\tilde{B}_Q$ (it suffices to take $\kappa_2=6/c$ with $c$ the constant that appears in the corkscrew condition).

Bearing in mind the previous considerations, we are ready to define our functions $b_Q$
as follows: we set $b_Q=\sigma(Q)\,\eta_Q\,k^{\hat{X}_Q}$ where $\eta_Q$ is a smooth cut-off (defined in $\re^{n+1}$) with $0\le \eta_Q\le 1$,
$\supp \eta_Q\subset 5\,\hat{B}_Q$, $\eta_Q\equiv 1$ in $4\,\hat{B}_Q$,
and $\|\nabla \eta_Q\|_\infty\lesssim \ell(Q)^{-1}$. We also take $q=\tilde{p}$ and we recall that $1<\tilde{p}\le p\le 2$.

Using Harnack's inequality, that $q=\tilde{p}$ and \eqref{eqn:scale-inva:converse-approx:wo-N} we obtain \eqref{Tb-bQ-Lp}:
\begin{multline*}
\int_{\pom} |b_Q|^q\,d\sigma
\le
\sigma(Q)^q\,\int_{5\,\hat{\Delta}_Q} \left(k^{\hat{X}_Q}\right)^q\,d\sigma
\lesssim
\sigma(Q)^q\,\int_{5\,\hat{\Delta}_Q} \left(k^{X_{5\hat{\Delta}_Q}}\right)^q\,d\sigma
\\
\lesssim
\tilde{C}\,\sigma(Q)^q\,\sigma(5\,\hat{\Delta}_Q)^{1-q}
\lesssim
\sigma(Q).
\end{multline*}
Regarding, \eqref{Tb-bQ-below} we have by Lemma \ref{Bourgainhm} and the Harnack chain condition that
$$
\left|\int_Q b_Q\,d\sigma\right|
=
\sigma(Q)\,\int_Q \eta_Q\,k^{\hat{X}_Q}\,d\sigma
\gtrsim
\sigma(Q)\,\omega^{X_Q}(\Delta_Q)
\ge
\sigma(Q)/C.
$$

Next we show \eqref{Tb-S-bQ}. Let $X\in \hat{B}_Q\cap \Omega_{ext}$. Then
\begin{align*}
\nabla_X^2\mathcal{S}b_Q(X)
&=
\int_{\pom}\nabla_X^2\mathcal{E}(X-y)\,b_Q(y)\,d\sigma(y)
\\
&
=
\sigma(Q)\,\int_{\pom}\nabla_X^2\mathcal{E}(X-y)\,\eta_Q(y)\,d\omega^{\hat{X}_Q}(y)
\\
&
=
\sigma(Q)\,\int_{\pom}\nabla_X^2\mathcal{E}(X-y)\,(\eta_Q(y)-1)\,d\omega^{\hat{X}_Q}(y)
\\
&\qquad\qquad
+
\sigma(Q)\,\int_{\pom}\nabla_X^2\mathcal{E}(X-y)\,d\omega^{\hat{X}_Q}(y)
\\
&
=
I_1(X)+I_2(X).
\end{align*}
For $I_2$ we observe that $u(Z)=\nabla^2_X \mathcal{E}(X-Z)$ is harmonic in $\Omega$ and $C^2(\bar{\Omega})$  since $X\in \Omega_{ext}$. Thus, for every $Z\in\Omega$ we have
\begin{equation}\label{eq2.17}
\nabla^2_X \mathcal{E}(X-Z)
=
u(Z)
=
\int_{\pom} u(y)\,d\omega^{Z}(y)
=
\int_{\pom}\nabla_X^2\mathcal{E}(X-y)\,d\omega^{Z}(y).
\end{equation}
We would like to point out that we have implicitly used uniqueness of the solution which follows from the maximum principle even for an unbounded domain since $\nabla^2_X \mathcal{E}(X-Z)\to 0$ as $Z\to \infty$.
We apply \eqref {eq2.17}
with $Z=\hat{X}_Q$ and note that since
$|X-\hat{X}_Q|\approx \ell(Q)$ for every  $X\in \hat{B}_Q\cap \Omega_{ext}$, we therefore obtain
$$
|I_2(X)|
=
\sigma(Q)|\nabla^2_X \mathcal{E}(X-\hat{X}_Q)|
\lesssim
\sigma(Q)\,|X-\hat{X}_Q|^{-(n+1)}
\approx
\ell(Q)^{-1}.
$$

For $I_1$, we observe that $\Omega$ is an approximating domain whose boundary consists of portions of faces of fattened Whitney cubes of size comparable to $2^{-N}$.  Thus, its (outward) unit normal $\nu$ is well defined a.e. on $\pom$, and we can apply the divergence theorem to obtain
\begin{align*}
I_1(X)
&
=
\sigma(Q)\,\nabla_X^2 \int_{\pom}\mathcal{E}(X-y)\,(\eta_Q(y)-1)\,d\omega^{\hat{X}_Q}(y)
\\
&=
\sigma(Q)\,\nabla_X^2 \int_{\pom}\mathcal{E}(X-y)\,(\eta_Q(y)-1)\,\nabla_Y G(\hat{X}_Q,y)\cdot \nu(y)\,d\sigma(y)
\\
&=
\sigma(Q)\,\nabla_X^2 \iint_{\Omega}\div_Y\big(\mathcal{E}(X-Y)\,(\eta_Q(Y)-1)\,\nabla_Y G(\hat{X}_Q,Y)\big)\,dY
\\
&=
\sigma(Q)\,\nabla_X^2 \iint_{\Omega}\nabla_Y\big(\mathcal{E}(X-Y)\big)\,(\eta_Q(Y)-1)\,\nabla_Y G(\hat{X}_Q,Y)\,dY
\\
&\qquad+
\sigma(Q)\,\nabla_X^2 \iint_{\Omega}\mathcal{E}(X-Y)\nabla\eta_Q(Y)\,\nabla_Y G(\hat{X}_Q,Y)\,dY
\\
&\qquad\qquad+
\sigma(Q)\,\nabla_X^2 \iint_{\Omega}\mathcal{E}(X-Y)(\eta_Q(Y)-1)\,\div_Y(\nabla_Y G(\hat{X}_Q,Y))\,dY
\\
&=:
I_{11}(X)+I_{12}(X)+0\,,
\end{align*}
where we have used that the term in the next-to-last line vanishes,
since $\eta_Q-1$ is supported in $\re^{n+1}\setminus 4\,\hat{B}_Q$,
and since $\hat{X}_Q\in \hat{B}_Q$ implies that $G(\hat{X}_Q,\cdot)$ is harmonic in $\Omega\setminus \hat{B}_Q$. Notice that the integration by parts can be justified even
when $\Omega$ is unbounded, since $\mathcal{E}$ and $G$ have sufficient decay at infinity.

We estimate the terms $I_{11}$ and $I_{12}$ in turn. First, notice that if $X\in \hat{B}_Q$ and $Y\in \Omega\setminus 4\,\hat{B}_Q$ we have $r(\hat{B}_Q)\lesssim |Y-x_Q|\approx |Y-\hat{X}_Q|\approx |Y-X|$. Then, writing $S_k(Q)=2^{k+1}\,\hat{B}_Q\setminus 2^{k}\,\hat{B}_Q$, $k\ge 2$, and using that
$\eta_Q-1$ is supported in $\re^{n+1}\setminus 4\,\hat{B}_Q$, we have that for every $X\in \hat{B}_Q$,
\begin{align}
|I_{11}(X)|
&\le
\sigma(Q)\,\iint_{\Omega}|\nabla_X^2\nabla_Y\mathcal{E}(X-Y)|\,|\eta_Q(Y)-1|\,|\nabla_Y G(\hat{X}_Q,Y)|\,dY \label{I11}
\\
&
\lesssim
\sigma(Q)\iint_{\Omega\setminus 4\,\hat{B}_Q} |Y-X|^{-(n+2)}\,|\nabla_Y G(\hat{X}_Q,Y)|\,dY\nonumber
\\
&
\lesssim
\sigma(Q)\sum_{k=2}^\infty \iint_{\Omega\cap S_k(Q)} |Y-x_Q|^{-(n+2)}\,|\nabla_Y G(\hat{X}_Q,Y)|\,dY \nonumber
\\
&
\lesssim
\ell(Q)^{-2}
\sum_{k=2}^\infty 2^{-(n+2)\,k}\iint_{\Omega\cap S_k(Q)} |\nabla_Y G(\hat{X}_Q,Y)|\,dY\nonumber
\\
&
=:
\ell(Q)^{-2}
\sum_{k=2}^\infty 2^{-(n+2)\,k}\mathcal{I}_k
. \nonumber
\end{align}
To estimate $\mathcal{I}_k$ we cover $S_k(Q)$ by a purely dimensional number of balls meeting $S_k(Q)$ and whose radii are $2^{k-5}\,r(\hat{B}_Q)$. Then, it suffices to get an estimate with the integral restricted to such a ball $B_k$. We may assume without loss of generality that $2\,B_k\not\subset \Omega_{ext}$ for otherwise we have that $\Omega\cap S_k(Q)\cap B_k=\emptyset$. We then have two cases: $2\,B_k\subset \Omega$ and $2\,B_k\not\subset \Omega$. In the second case,
since $2\,B_k$ is neither contained in $\Omega$ nor in $\Omega_{ext}$, there exists
$y_k\in \pom\cap 2\,B_k$.  We then set $\tilde{B}_k=B(y_k,3\,r(B_k))$, and observe that this ball is centered on $\pom$ and contains $B_k$.  Since
$G(\hat{X}_Q,\cdot)$ is harmonic in $2\tilde{B}_k\cap\Omega$ and vanishes
on $\pom$, and since $\Omega$ is an approximating domain in which the Gauss/Green theorem
holds, we can use Caccioppoli's inequality at the boundary
to write
\begin{align*}
\mathcal{I}_{k,B_k}
&
:=
\iint_{\Omega\cap S_k(Q)\cap B_k} |\nabla_Y G(\hat{X}_Q,Y)|\,dY
\\
&\le
\iint_{\Omega\cap \tilde{B}_k} |\nabla_Y G(\hat{X}_Q,Y)|\,dY
\\
&
\lesssim
|\tilde{B}_k|^{\frac12}
\,
\left(\iint_{\Omega\cap \tilde{B}_k} |\nabla_Y G(\hat{X}_Q,Y)|^2\,dY\right)^{1/2}
\\
&
\lesssim
(2^k\,\ell(Q))^{\frac{n+1}2}
\,
r(\tilde{B}_k)^{-1}\,
\left(\iint_{\Omega\cap 2\,\tilde{B}_k} |G(\hat{X}_Q,Y)|^2\,dY\right)^{1/2}
\\
&
\lesssim
(2^k\,\ell(Q))^{\frac{n+1}2-1}
(2^k\,\ell(Q))^{1-n}\,|\tilde{B}_k|^{\frac12}
\\
&
=2^k\,\ell(Q)\,,
\end{align*}
by the size estimates for the Green function plus the fact that $|\hat{X}_Q-Y|\ge 2^{k-2}\,r(\hat{B}_Q)$ for $Y\in \Omega\cap 2\,\tilde{B}_k$.

In the other case, i.e., when $2\,B_k\subset \Omega$, the situation is simpler:  we just use the (interior) Caccioppoli inequality and repeat the same computations with $B_k$ in place of $\tilde{B}_k$.  Summing over the collection of $B_k$'s covering $S_k(Q)$, which for fixed $k$ is a family of uniformly bounded cardinality, we conclude that
\begin{equation}\label{Ik}
\mathcal{I}_k
\lesssim
2^k\,\ell(Q).
\end{equation}
We plug this estimate into \eqref{I11} to obtain
$$
|I_{11}(X)|
\lesssim
\ell(Q)^{-2}\sum_{k=2}^\infty 2^{-(n+2)\,k}\mathcal{I}_k
\lesssim
\ell(Q)^{-1}.
$$

Let us estimate $I_{12}(X)$. We notice that $\nabla\eta_Q$ is supported in $5\,\hat{B}_Q\setminus 4\,\hat{B}_Q$. Take $Y\in 5\,\hat{B}_Q\setminus 4\,\hat{B}_Q$. If $X\in \hat{B}_Q$ then $|X-Y|\approx |\hat{X}_Q-Y|\approx \ell(Q)$ which yields $|\nabla_X^2 \mathcal{E}(X-Y)|\lesssim|X-Y|^{-(n+1)}\approx \ell(Q)^{-(n+1)}$. Then using \eqref{Ik} with $k=2$ we obtain
$$
|I_{12}(X)|
\lesssim
\ell(Q)^{-2}\iint_{\Omega\cap5\,\hat{B}_Q\setminus 4\,\hat{B}_Q} |\nabla_Y G(\hat{X}_Q,Y)|\,dY
\lesssim
\ell(Q)^{-2}\mathcal{I}_2
\lesssim
\ell(Q)^{-1}.
$$

Collecting our estimates for $I_{11}$ and $I_{12}$ we have shown that $|I_1(X)|\lesssim \ell(Q)^{-1}$. Thus, for all $X\in \hat{B}_Q\cap\Omega_{ext}$ we have
$$
|\nabla_X^2\mathcal{S}b_Q(X)|
\le
C\,\ell(Q)^{-1}.
$$

\begin{remark}\label{remark:I1}
In the previous argument, for the estimate of $I_1$ we have only used that $X\in\hat{B}_Q$ and we have not used that $X\in\Omega_{ext}$.
\end{remark}

Next we let $X\in \hat{B}_Q\cap\Omega$, and suppose first that
$\delta(X)\gtrsim \ell(Q)$. If $y\in 5\hat{\Delta}_Q$ we have
$|\nabla_X^2\mathcal{E}(X-y)|\lesssim |X-y|^{-(n+1)}\le \delta(X)^{-(n+1)}\lesssim \ell(Q)^{-(n+1)}$.
Consequently,
$$
|\nabla_X^2\mathcal{S}b_Q(X)|
=
\left|\int_{\pom}\nabla_X^2\mathcal{E}(X-y)\,b_Q(y)\,d\sigma(y)\right|
\lesssim
\ell(Q)^{-1}  \omega^{\hat{X}_Q}(5\,\hat{\Delta}_Q)
\le
\ell(Q)^{-1}.
$$

It remains to treat the case $X\in \hat{B}_Q\cap\Omega$, with $\delta(X)\le c \ell(Q)$, where $c$ is to be chosen. Notice that
$$
2\,\ell(Q)/C\le \delta(\hat{X}_Q)\le
|\hat{X}_Q-X|+\delta(X)
\le
|\hat{X}_Q-X|+c\,\ell(Q).
$$
Then we pick $c<C^{-1}$ we obtain $|\hat{X}_Q-X|>\ell(Q)/C$. Next, we write as before $\nabla_X^2\mathcal{S}b_Q(X)=I_1(X)+I_2(X)$. For $I_1$, since $X\in \hat{B}_Q$, by Remark \ref{remark:I1} we conclude that $|I_1(X)|\lesssim \ell(Q)^{-1}$.
For $I_2$ we observe that \eqref{eq2.greendef} gives
\begin{multline*}
|I_2(X)|
=
\sigma(Q)\,\left|\int_{\pom}\nabla_X^2\mathcal{E}(X-y)\,d\omega^{\hat{X}_Q}(y)\right|
=
\sigma(Q)\,\big|\nabla_X^2 \big(\mathcal{E}(X-\hat{X}_Q)-G(X,\hat{X}_Q)\big)\big|
\\
\lesssim
\sigma(Q)\,|X-\hat{X}_Q|^{-(n+1)}+\sigma(Q)\,|\nabla_X^2 G(X,\hat{X}_Q)|
\lesssim
\ell(Q)^{-1}+\sigma(Q)\,|\nabla_X^2 G(X,\hat{X}_Q)|.
\end{multline*}
Therefore, collecting all of our estimates, we have shown that for every $X\in \hat{B}_Q$,
\begin{equation}\label{estimate-sf}
|\nabla_X^2\mathcal{S}b_Q(X)|
\le
C\,\ell(Q)^{-1}+\sigma(Q)\,|\nabla_X^2 G(X,\hat{X}_Q)|\,\chi_{\hat{B}_Q\cap\Omega\cap\{\delta(X)\le c\ell(Q)\}}(X).
\end{equation}

Notice that if $x\in \pom$ then
\begin{multline*}
\iint_{\tilde{\Lambda}_{Q}(x)}\,\frac{dY}{\delta(Y)^{n-1}}
\le
\sum_{x\in Q'\in\dd_Q}\,\sum_{I\in \W_{Q'}\cup\W_{Q'}^{ext}} \iint_{I}\,\frac{dY}{\delta(Y)^{n-1}}
\\
\approx
\sum_{x\in Q'\in\dd_Q}\,\sum_{I\in \W_{Q'}\cup\W_{Q'}^{ext}} \frac{\ell(I)^{n+1}}{\ell(I)^{n-1}}
\lesssim
\sum_{x\in Q'\in\dd_Q}\,\ell(Q')^2
\lesssim
\ell(Q)^2
\end{multline*}
where we have used that $\# (\W_{Q'}\cup\W_{Q'}^{ext})\lesssim C_{C_0}$. Therefore, using \eqref{local-cone-ball} and \eqref{estimate-sf}
we have
\begin{align*}
&\int_{Q}\left(\iint_{\tilde{\Lambda}_{Q}(x)} |\nabla^2 \mathcal{S}b_Q(Y)|^2 \,\frac{dY}{\delta(Y)^{n-1}}
\right)^{\frac{q}2}\,d\sigma(x)
\\
&\lesssim
\ell(Q)^{-q}
\int_{Q}\left(\iint_{\tilde{\Lambda}_{Q}(x)}\,\frac{dY}{\delta(Y)^{n-1}}
\right)^{\frac{q}2}\,d\sigma(x)
\\
&\qquad+
\sigma(Q)^q\,
\int_{Q}\left(\iint_{\Lambda_{Q}(x)\cap \{\delta(Y)\le c\ell(Q)\}} |\nabla_Y^2 G(Y,\hat{X}_Q)|^2\,\frac{dY}{\delta(Y)^{n-1}}
\right)^{\frac{q}2}\,d\sigma(x)
\\
&\lesssim
\sigma(Q)
+
\sigma(Q)^q\,
\int_{Q} S_{Q} u(x)^q\,d\sigma(x)
\end{align*}
where we set $u(Y)=\nabla_Y G(Y,\hat{X}_Q)$
and
$$
S_{Q}u(x)
=
\left(\iint_{\Gamma_{Q}(x)} |\nabla u(Y)|^2\,\frac{dY}{\delta(Y)^{n-1}}\right)^{\frac12},
\qquad
\Gamma_{Q}(x):=  \bigcup_{x\in Q'\in\dd_{Q}} U_{Q'}.
$$
We claim that
\begin{equation}\label{desired-S-Green}
\int_{Q} S_{Q} u(x)^q\,d\sigma(x)
\lesssim
\sigma(Q)^{1-q}.
\end{equation}
Assuming this momentarily we obtain as desired \eqref{Tb-S-bQ}:
$$
\int_{Q}\left(\iint_{\tilde{\Lambda}_{Q}(x)} |\nabla^2 \mathcal{S}b_Q(Y)|^2 \,\frac{dY}{\delta(Y)^{n-1}}
\right)^{\frac{q}2}\,d\sigma(x)
\lesssim
\sigma(Q).
$$
Modulo the claim \eqref{desired-S-Green}, we have now verified all the hypotheses of
Theorem \ref{theor:Tb} and we therefore conclude that \eqref{Tb-conclusion} holds.
The latter in turn implies that $\pom$ is UR (see \cite[p. 44]{DS2}).

To complete this stage of our proof, it remains to establish \eqref{desired-S-Green}.  We do this
in Section \ref{section:proof-S}, after first proving a bound for the square function in terms of the non-tangential maximal function in Section \ref{section:good-lambda}.
As mentioned before, at present, $\Omega$ is actually an approximating domain (i.e., $\Omega$ stands for $\Omega_N$ with $N$ large enough). Thus, the conclusion that we have obtained in the current step is that for every $N\gg 1$, we have that $\Omega_N$ has the UR property with uniform constants. In
Section \ref{section:UR-general} we show that this property may be transmitted to $\Omega$,
thereby completing the proof of Theorem \ref{theorem:converse}.

\subsection{Step 3: Good-$\lambda$ inequality for the square function and the non-tangential maximal function}\label{section:good-lambda}
We recall that $U_Q=\bigcup_{\mathcal{W}_Q^*} I^*$ with  $I^*=(1+\lambda)I$,
where $0<\lambda\le \lambda_0/2$ was fixed above. As mentioned in Remark \ref{remark:lambda0}, since $0<2\,\lambda\le \lambda_0$, the fattened Whitney boxes $I^{**}=(1+2\,\lambda)I$,
and corresponding Whitney regions $U_Q^*:=\bigcup_{\mathcal{W}_Q^*} I^{**}$ enjoy the same properties as do $I$ and $U_Q$.

Let us set some notation:  for fixed $Q_0\in\dd$,  and for $x\in Q_0$,
$$
S_{Q_0}u(x)
=
\left(\iint_{\Gamma_{Q_0}(x)} |\nabla u(Y)|^2\,\frac{dY}{\delta(Y)^{n-1}}\right)^{\frac12},
\qquad\qquad
\tilde{N}_{Q_0,*} u(x)
=
\sup_{Y\in \tilde{\Gamma}_{Q_0}(x)} |u(Y)|
$$
where
$$
\Gamma_{Q_0}(x):=  \bigcup_{x\in Q\in\dd_{Q_0}} U_Q,
\qquad\qquad
\tilde{\Gamma}_{Q_0}(x):= \bigcup_{x\in Q\in\dd_{Q_0}} U_Q^*.
$$
We also define a localized dyadic maximal operator
$$
\mathcal{M}_{Q_0} f(x)
=
\sup_{x\in Q\in\dd_{Q_0}} \fint_{Q} |f(y)|\,d\sigma(y).
$$
Given any $Q\in\dd_{Q_0}$,  for any family $\F\subset\dd_{Q_0}$ of disjoint dyadic cubes we define the fattened versions of the Carleson box $T_Q$ and the local sawtooth  $\Omega_{\F,Q_0}$ by
$$
\tilde{T}_{Q_0}=
{\rm int } \left(\bigcup_{Q\in\dd_{Q_0}} U_Q^*\right),
\qquad
\tilde{\Omega}_{\F,Q_0}
=
{\rm int } \left(\bigcup_{Q\in\dd_{\F,Q_0}} U_Q^*\right).
$$
It is straightforward to show that there exists $\kappa_0$ (depending on $K_0$) large enough such that
\begin{equation}\label{condition:kappa0}
\overline{\tilde{T}_{Q_0}}\subset B(x_{Q_0}, \kappa_0\,\ell(Q_0))\cap\overline{\Omega}.
\end{equation}
Let us write $B^{\kappa_0}_{Q_0}=B(x_{Q_0}, \kappa_0\,\ell(Q_0))$ and note that in particular, for every $x\in Q_0$ and every pairwise disjoint family
$\F\subset\dd_{Q_0}$, we have
\begin{equation}\label{condition:kappa0-cones-sawtooth}
\overline{\Gamma_{Q_0}(x)}\subset
\overline{\tilde{\Gamma}_{Q_0}(x)}
\subset
\overline{\tilde{T}_{Q_0}}\subset
B^{\kappa_0}_{Q_0}\cap\overline{\Omega}\,,
\qquad
\overline{\tilde{\Omega}_{\F,Q_0}}\subset \overline{\tilde{T}_{Q_0}}\subset  B^{\kappa_0}_{Q_0}\cap\overline{\Omega}.
\end{equation}

\begin{proposition}\label{prop:good-lambda}
Given $Q_0\in\dd$, let $u$ be harmonic in $2B^{\kappa_0}_{Q_0}\cap\Omega$. Then, for every $1<q<\infty$ we have
\begin{equation}\label{Sf-NT}
\|S_{Q_0} u\|_{L^q(Q_0)}
\le
C\,\|\tilde{N}_{Q_0,*} u\|_{L^q(Q_0)},
\end{equation}
provided the left hand side is finite.
\end{proposition}

\noindent{\it Remark.} Let us emphasize that in this part of the argument, $u$ is allowed to be any
harmonic function in $2\,B^{\kappa_0}_{Q_0}\cap\Omega$, not only $\nabla G(\cdot,\hat{X}_Q)$
as above.

\begin{proof}
The proof is based on the standard ``good-$\lambda$'' argument of \cite{DJK}
(see also \cite{Br}), but adjusted to our setting.  Fix $1<q<\infty$ and assume (qualitatively)  that the left hand side of \eqref{Sf-NT} is finite. In particular we have that
$$
\fint_{Q_0} S_{Q_0} u\,d\sigma
\le
\left(\fint_{Q_0} (S_{Q_0} u)^q\,d\sigma\right)^{\frac1q}
<\infty.
$$
We consider first the case that $\lambda>\fint_{Q_0} S_{Q_0} u\,d\sigma$.
Set $$E_\lambda=\{x\in Q_0: \mathcal{M}_{Q_0}(S_{Q_0} u)(x)>\lambda\}\,.$$
Then by the usual Calder\'{o}n-Zygmund decomposition argument,
there exists a pairwise disjoint family of cubes $\{P_j\}_j\subset\dd_{Q_0}\setminus\{Q_0\}$,
which are maximal with respect to the property that
$\fint_{P_j} S_{Q_0} u\,d\sigma >\lambda$, such that $E_\lambda=\cup_j P_j$.
Let $\tilde{P}_j$ denote the dyadic parent of $P_j$. Then
$\tilde{P}_j$ must contain a subset of positive $\sigma$-measure
on which $S_{Q_0}u(x)\le \lambda$, by maximality of $P_j$.
Fix a cube in $\{P_j\}_j$, say $P_0$, and set
$$
F_\lambda
=
\{x\in P_0: S_{Q_0} u(x)>\beta\,\lambda, \tilde{N}_{Q_0,*} u(x) \le \gamma\,\lambda\}
$$
where $\beta>1$ is to be chosen and $0<\gamma<1$. We are going to show that
$$
\sigma(F_\lambda)\le C\,\gamma^{\theta}\,\sigma(P_0).
$$
Assume that $\sigma(F_\lambda)>0$ otherwise there is nothing to prove. We claim
that if $\beta$ is large enough, depending only on the ADR constants,
then $\sigma(F_\lambda)<\sigma(P_0)$. Otherwise, we would have $S_{Q_0} u(x)>\beta\,\lambda$ for a.e. $x\in P_0$ and then by the maximality of $P_0$ and the fact that $\pom$ is ADR,
$$
\beta\,\lambda
<
\fint_{P_0} S_{Q_0} u d\sigma
\le
C_1\,\fint_{\tilde{P}_0} S_{Q_0} u d\sigma
\le C_1\,\lambda.
$$
Choosing $\beta>C_1$, we obtain a contradiction.
Thus,  $\sigma(F_\lambda)<\sigma(P_0)$, so by the inner regularity of $\sigma$,
there exists a compact set $F\subset F_\lambda\subset P_0$ such that
$$
0<\frac12\,\sigma(F_\lambda)\le \sigma(F)\le \sigma(F_\lambda)<\sigma(P_0).
$$
Since $\sigma(F)<\sigma(P_0)$, it follows that  ${\rm int}(P_0)\setminus F$ is a non-empty open set.
By a standard stopping time procedure, we may then subdivide $P_0$ dyadically, to extract a pairwise disjoint family
of cubes $\F=\{Q_j\}_j\subset\dd_{P_0}\setminus\{P_0\}$, which are maximal with respect to the property that $Q_j\cap F=\emptyset$.

We see next that $F=P_0\setminus(\cup_\F Q_j)$.  If $x\in F\subset P_0$ and $x\in Q_j$ then $x\in Q_j\cap F$ which contradicts the fact that $Q_j\cap F=\emptyset$. On the other hand, let $x\in P_0\setminus(\cup_\F Q_j)$. Pick $Q_k^x\in\dd_{P_0}$ the unique cube containing $x$ with $\ell(Q_k^x)=2^{-k}\,\ell(P_0)$, $k\ge 1$. Note that $Q_k^x\cap F\neq\emptyset$, otherwise $Q_k^x\subset Q_j$ for some $j$ and we would have $x\in Q_j$. Let $x_k\in Q_k^x\cap F$. Then $|x_k-x|\lesssim \ell(Q_k^x)=2^{-k}\,\ell(P_0)$ and therefore $x_k\to x$ as $k\to\infty$. Since $F$ is closed and $x_k\in F$ we conclude that $x\in F$ as desired.

Notice that for any $x\in F=P_0\setminus(\cup_\F Q_j)$ and for any $z\in \tilde{P}_0$ we have that
\begin{multline*}
\Gamma_{Q_0}(x)
=
\bigcup_{x\in Q\in\dd_{Q_0}} U_Q
=
\bigg(\bigcup_{x\in Q\in\dd_{Q_0}, Q\subset P_0} U_Q\bigg)
\bigcup
\bigg(\bigcup_{x\in Q\in\dd_{Q_0}, Q\supsetneq P_0} U_Q\bigg)
\\
=
\Gamma_{P_0}(x)\bigcup
\bigg(\bigcup_{Q\in\dd_{Q_0}, Q\supset \tilde{P}_0} U_Q\bigg)
\subset
\Gamma_{P_0}(x)\bigcup
\bigg(\bigcup_{z\in Q\in\dd_{Q_0}} U_Q\bigg)
=
\Gamma_{P_0}(x)\cup \Gamma_{Q_0}(z).
\end{multline*}
Let us recall that $P_0$ is a Calder\'{o}n-Zygmund cube in $E_\lambda$ and that we can pick $z_0\in \tilde{P}_0$ with $S_{Q_0} u(z_0)\le \lambda$ (indeed we know that this happens in a set of positive measure in $\tilde{P}_0$.) Then for any $x\in F\subset F_\lambda$ we have
\begin{align}\label{local-Sf}
\beta\,\lambda
&
<
S_{Q_0} u(x)
=
\left(\iint_{\Gamma_{Q_0}(x)} |\nabla u(Y)|^2\,\frac{dY}{\delta(Y)^{n-1}}\right)^{\frac12}
\\
&
\le
\left(\iint_{\Gamma_{P_0}(x)} |\nabla u(Y)|^2\,\frac{dY}{\delta(Y)^{n-1}}\right)^{\frac12}
+
\left(\iint_{\Gamma_{Q_0}(z_0)} |\nabla u(Y)|^2\,\frac{dY}{\delta(Y)^{n-1}}\right)^{\frac12}
\nonumber
\\
&=
S_{P_0} u(x)+ S_{Q_0}u(z_0)
\nonumber
\\
&\le
S_{P_0} u(x)+\lambda \nonumber
\end{align}
and therefore $ F\subset \{x\in P_0: S_{P_0} u(x)>(\beta-1)\,\lambda\}$.

Next we claim that
\begin{equation}\label{cones-sawtooth}
\bigcup_{x\in F} \Gamma_{P_0}(x)
\,\subset\,
\bigcup_{Q\in\dd_{\F,P_0}} U_Q
\,\subset\,
\tilde{\Omega}_{\F,P_0}
\end{equation}
where $\tilde{\Omega}_{\F,P_0}$ is the fattened version of $\Omega_{\F,P_0}$ defined above.
The second containment in \eqref{cones-sawtooth} is trivial (since $U_Q\subset U_Q^*$),
so let us verify the first.
We take $Y\in \Gamma_{P_0}(x)$ with $x\in F$. Then, $Y\in U_Q$  where $x\in Q\in\dd_{P_0}$. Since $x\in F=P_0\setminus(\cup_\F Q_j)$ we must have $Q\in \dd_\F$ (otherwise $Q\subset Q_j$ for some $Q_j\in\F$ and this would imply that $x\in Q_j$) and therefore $Q\in \dd_{\F,P_0}$ which gives the first inclusion.

Write $\tomt=\tomt^{X_0}$ to denote the harmonic measure for the domain $\tilde{\Omega}_{\F,P_0}$ where $X_0=A_{P_0}$ is given in \cite[Proposition 6.4]{HM-I} and \cite[Corollary 6.6]{HM-I}
(which we may apply to $\tilde{\Omega}_{\F,P_0}$ in place of
$\Omega_{\F,P_0}$
since $0<2\,\lambda\le\lambda_0$, see Remark \ref{remark:lambda0}).
Let us also write $\tilde{\delta}_\star(Y)$ to denote the distance from $Y$ to $\partial\tilde{\Omega}_{\F,P_0}$, and $\tilde{G}_\star$ to denote the corresponding Green function.   Given $Y\in \tilde{\Omega}_{\F,P_0}$, let us choose $y_Y\in
\partial\tilde{\Omega}_{\F,P_0}$ such that $|Y-y_Y|= \tilde{\delta}_\star(Y)$.
By definition, for $x\in F$ and $Y\in\Gamma_{P_0}(x)$,
there is a $Q\in\dd_{P_0}$ such that $Y\in U_Q$ and $x\in Q$.
Thus, by the triangle inequality, and the definition of $U_Q$, we have that for $Y\in\Gamma_{P_0}(x)$,
\begin{equation}\label{eq2.29}
|x-y_Y| \leq |x-Y|+  \tilde{\delta}_\star(Y) \approx \delta(Y) +  \tilde{\delta}_\star(Y) \approx
\tilde{\delta}_\star(Y)\,,
\end{equation}
where in the last step we have used that
\begin{equation}\label{eq2.30}
\delta(Y) \approx  \tilde{\delta}_\star(Y) \quad {\rm for}\,\,\, Y\in\bigcup_{Q\in\dd_{\F,P_0}} U_Q\,.
\end{equation}
Then,  since $F=P_0\setminus(\cup_\F Q_j)\subset\pom\cap\partial \tilde{\Omega}_{\F,P_0}$, see  \cite[Proposition 6.1]{HM-I},
we have
\begin{align}\label{sf-L2}
&\tomt(F)\,(\lambda\,(\beta-1))^2
=
\int_F (\lambda\,(\beta-1))^2\,d\tomt
\le
\int_F S_{P_0} u(x)^2\,d\tomt(x)
\\
&
\qquad=
\int_F \iint_{\Gamma_{P_0}(x)} |\nabla u(Y)|^2\,\frac{dY}{\delta(Y)^{n-1}}\,d\tomt(x)
\nonumber
\\
&
\qquad\lesssim
\iint_{\bigcup\limits_{Q\in\dd_{\F,P_0}}\!\!\!\! U_Q} |\nabla u(Y)|^2\,\delta(Y)\, \left(\frac1{\delta(Y)^n}\,\int_{F\cap B(y_Y,C\,\tilde{\delta}_\star(Y))} \,d\tomt(x) \right)\,dY
\nonumber
\\
&
\qquad\lesssim
\iint_{\tilde{\Omega}_{\F,P_0}} |\nabla u(Y)|^2\,\tilde{\delta}_\star(Y)\,\frac{\tomt^{X_0}(\tilde{\Delta}_\star(y_Y,C\,\tilde{\delta}_\star(Y))}{(C\,\tilde{\delta}_\star(Y))^n}\,dY,
\nonumber
\end{align}
where we have used  \eqref{eq2.29}, \eqref{cones-sawtooth}and \eqref{eq2.30}.

We now claim that for $Y\in \tilde{\Omega}_{\F,P_0}$, we have
\begin{equation}\label{eq2.32}
\frac{\tomt^{X_0}(\tilde{\Delta}_\star(y_Y,C\,\tilde{\delta}_\star(Y))}{(C\,\tilde{\delta}_\star(Y))^n}
\lesssim
\frac{\tilde{G}_\star(X_0,Y)}{\tilde{\delta}_\star(Y)}.
\end{equation}
Indeed,  if $\tilde{\delta}_\star(Y)< \tilde{\delta}_\star(X_0)/(2\,C),$ then \eqref{eq2.32} is immediate by Lemma \ref{lemma2.cfms}  and \eqref{eq2.green4}.   Otherwise, we have
$ \tilde{\delta}_\star(Y)\approx \tilde{\delta}_\star(X_0)\approx\ell(P_0)\gtrsim |X_0-Y|,$
whence \eqref{eq2.32} follows directly from \eqref{eq2.green2} and the Harnack Chain condition,
and the fact that harmonic measure is a probability measure.

We recall that by hypothesis, $u$ is harmonic in $2B^{\kappa_0}_{Q_0}\cap\Omega\supset\tilde{\Omega}_{\F,P_0}$ (cf. \eqref{condition:kappa0-cones-sawtooth}).
Let $\mathcal{L} := \nabla\cdot\nabla$ denote the usual Laplacian in $\ree$, so that
$\mathcal{L} (u^2)= 2\,|\nabla u|^2$ in
$\tilde{\Omega}_{\F,P_0}$.
Combining these observations with
\eqref{sf-L2}-\eqref{eq2.32}, we see that
\begin{align}\label{eq2.33}
\tomt(F)\,(\lambda\,(\beta-1))^2
&\lesssim
\iint_{\tilde{\Omega}_{\F,P_0}} |\nabla u(Y)|^2\,\tilde{G}_\star(X_0,Y)\,dY
\\[4pt]
&=\frac12
\iint_{\tilde{\Omega}_{\F,P_0}} \mathcal{L}(u^2)(Y)\,\tilde{G}_\star(X_0,Y)\,dY \nonumber
\\[4pt]
&=-\frac12 u(X_0)^2+ \frac12
\int_{\partial\tilde{\Omega}_{\F,P_0}} u(y)^2 d\tomt^{X_0}(y). \nonumber
\end{align}
where the last step is a well known identity obtained using properties of the Green function.

Let $Y\in\tilde{\Omega}_{\F,P_0}$, so that $Y\in U_Q^*$ for some $Q\in\dd_{\F,P_0}$. By definition of $\dd_{\F,P_0}$, this $Q$ cannot be contained in any $Q_j\in\F$.  Therefore, $Q\cap F\neq\emptyset$.
Indeed, otherwise $Q\cap F=\emptyset$, which by maximality of the cubes in $\F$, would imply that $Q\subset Q_j$ for some $Q_j\in\F$, a contradiction. Thus, there is some $x\in Q\cap F$ which then satisfies $x\in P_0\setminus (\cup_\F Q_j)$ and $x\in Q\in\dd_{P_0}$. Hence, $Y\in U_Q^*\subset\tilde{\Gamma}_{P_0}(x)$ with $x\in F$. Since  $F\subset F_\lambda$ we have that
$$
|u(Y)|
\le
\sup_{Z\in \tilde{\Gamma}_{P_0}(x)}|u(Z)|
\le
\tilde{N}_{Q_0,*} u(x)\le \gamma\,\lambda.
$$
Thus, $|u(Y)|\le \gamma\,\lambda$ and in particular $u(Y)\ge -\gamma\,\lambda$, for all $Y\in\tilde{\Omega}_{\F,P_0}$. Next, we apply \cite[Theorem 6.4]{JK} to the (qualitative)
NTA domain $\tilde{\Omega}_{\F,P_0}$
(we recall that, in the present stage of the argument,
$\Omega$ is actually $\Omega_N$ for some large $N$, which satisfies the qualitative
exterior Corkscrew condition (Definition \ref{qualextCS});  thus, the bounded domain
$\tilde{\Omega}_{\F,P_0}$ enjoys an
exterior Corkscrew condition with constants that may depend very badly on $N$.)
Of course, the interior Corkscrew and Harnack chain constants, as well as the ADR constants,
are controlled uniformly in $N$).
Consequently,
we obtain that $u$ has non-tangential limit $\tomt$-a.e.. Since $|u(Y)|\le \gamma\,\lambda$ for every $Y\in\tilde{\Omega}_{\F,P_0}$, its non-tangential limit therefore
satisfies this same bound $\tomt$-a.e., i.e., for $\tomt$-a.e. $y\in \partial\tilde{\Omega}_{\F,P_0}$ we have $|u(y)|\le \gamma\,\lambda$. Consequently, by \eqref{eq2.33}, we conclude that
$
\tomt(F)\,(\lambda\,(\beta-1))^2
\lesssim
(\gamma\,\lambda)^2
$
and then $\tomt(F)\lesssim\gamma^2$.

We use the notation of \cite[Lemma 6.15]{HM-I} (with $\tilde{\Omega}_{\F,P_0}$ and $\tomt^{X_0}$ in place of $\Omega_{\F,Q_0}$ and $\hm_\star^{X_0}$). Since $P_0$ is not contained in any $Q_j\in\F$,
\begin{align*}
\P_\F\nu(P_0)
&=
\tomt^{X_0}\Big(P_0\setminus (\cup_\F Q_j)\Big) +\sum_{Q_j\in\F, Q_j\subsetneq Q} \frac{\sigma(P_0\cap Q_j)}{\sigma(Q_j)}\,
\tomt^{X_0}(P_j)
\\
&=
\tomt^{X_0} \Big(P_0\setminus (\cup_\F Q_j)\Big) +\sum_{Q_j\in\F, Q_j\subsetneq P_0} \tomt^{X_0}(P_j)
\\
&
\gtrsim
\tomt^{X_0}\Big(P_0\setminus (\cup_\F Q_j)\Big) + \sum_{Q_j\in\F, Q_j\subsetneq Q}
\tomt^{X_0}\left(B(x_j^\star,r_j)\cap \partial \Omega_{\F,P_0}\right)
\\
&
\geq
\tomt^{X_0}(
\tilde{\Delta}_\star^{P_0}),
\end{align*}
where in the third line we have used the doubling property of $\tomt^{X_0}$ (plus a subdivision and Harnack Chain argument if $\ell(Q_j)\approx \ell(P_0)$), and in the last line we have used
\cite[Proposition 6.12]{HM-I}, along with
\cite[Proposition 6.1]{HM-I} and \cite[Proposition 6.3]{HM-I} and the doubling property to ignore
the difference between $Q\setminus (\cup_\F Q_j)$ and $Q\cap\partial\Omega_{\F,Q_0}$. Also,
$\tilde{\Delta}_\star^{P_0}=\tilde{\Delta}_\star^{P_0}(\tilde{x}_{P_0}^\star,t_{P_0})$ is a surface ball such that
$\tilde{x}_{P_0}^\star\in \partial \tilde{\Omega}_{\F,P_0}$,
$t_{P_0}\approx \ell(P_0)$, $\dist(P_0, \tilde{\Delta}_\star^{P_0})\lesssim \ell(P_0)$, and the implicit constants may depend upon $K_0$, see \cite[Proposition 6.12]{HM-I}. Using Lemma \ref{Bourgainhm}, Harnack chain and \cite[Corollary 6.6]{HM-I} it is immediate to show that $\tomt^{X_0}(
\tilde{\Delta}_\star^{P_0})\ge C$ and therefore $\P_\F\nu(P_0)\ge C$. On the other hand, since $F=P_0\setminus (\cup_\F Q_j)$ we conclude that
$$
\frac{\P_\F \nu(F)}{\P_\F \nu(P_0)}
\lesssim
\P_\F \nu(F)
=
\tomt^{X_0}(F)
\lesssim
\gamma^2
.
$$
Thus Harnack chain, \cite[Corollary 6.6]{HM-I} and \cite[Lemma 6.15]{HM-I} imply
\begin{equation}\label{smallness-omega}
\omega^{X_{\hat{\Delta}_{P_0}}}(F)\le \frac{\omega^{X_{\hat{\Delta}_{P_0}}}(F)}{\omega^{X_{\hat{\Delta}_{P_0}}}(P_0)}
\approx
\frac{\omega^{X_0}(F)}{\omega^{X_0}(P_0)}
=
\frac{\P_\F \omega^{X_0}(F)}{\P_\F \omega^{X_0}(P_0)}
\lesssim
\left(\frac{\P_\F \nu(F)}{\P_\F \nu(P_0)}\right)^{\frac1{\theta}}
\lesssim
\gamma^{2/\theta},
\end{equation}
where we recall that $\Delta_{P_0}\subset P_0\subset \hat{\Delta}_{P_0}$.

We need the following auxiliary result, the proof is given below.

\begin{lemma}\label{lemma:self-impro-RHp}
Assume \eqref{eqn:scale-inva:converse-approx:wo-N}.
Then, given a surface ball $\Delta_0=B_0\cap\pom$ for every $\Delta=B\cap \pom$ with $B\subset B_0$ we have
\begin{equation}\label{RHP-rest}
\left(\fint_{\Delta} \left(k^{X_{\Delta_0}}\right)^{\tilde{p}} d\sigma\right)^{1/\tilde{p}} \leq C\,\fint_{\Delta}k^{X_{\Delta_0}}\,d\sigma.
\end{equation}
\end{lemma}

Notice that this result says that $k^{X_{\Delta_0}}\in RH_{\tilde{p}}(\Delta_0)\subset A_\infty(\Delta_0)$ for all $\Delta_0$ and the constants are uniform in $\Delta_0$. Since $A_\infty(\Delta_0)$ defines an equivalence relationship
(on the set of doubling measures on $\Delta_0$)
we obtain that $\sigma\in A_\infty(\Delta_0,\omega^{X_{\Delta_0}})$ for all $\Delta_0$ and the constants are uniform in $\Delta_0$. Thus, there exist positive constants $C$ and $\vartheta$ such that
for every $\Delta_0=B_0\cap\pom$, $\Delta=B\cap \pom$ with $B\subset B_0$ and every Borel set $E$ we have
$$
\frac{\sigma(E)}{\sigma(\Delta)}
\le
C\,
\left(\frac{\omega^{X_{\Delta_0}}(E)}{\omega^{X_{\Delta_0}}(\Delta)} \right)^{\vartheta}
$$
We apply this with $\Delta=\Delta_0=\hat{\Delta}_{P_0}$ and with $E=F\subset P_0\subset \hat{\Delta}_{P_0}$. Then,
Lemma \ref{Bourgainhm} and \eqref{smallness-omega} imply
$$
\frac{\sigma(F)}{\sigma(P_0)}
\approx
\frac{\sigma(F)}{\sigma(\hat{\Delta}_{P_0})}
\lesssim
\left(\frac{\omega^{X_{\hat{\Delta}_{P_0}}}(F)}{\omega^{X_{\hat{\Delta}_{P_0}}}(\hat{\Delta}_{P_0})} \right)^{\vartheta}
\lesssim
\omega^{X_{\hat{\Delta}_{P_0}}}(F)^{\vartheta}
\lesssim
\gamma^{2\,\vartheta/\theta}.
$$
Let us recall that $\sigma(F_\lambda)\le 2\,\sigma(F)$ and therefore we have obtained
\begin{equation}\label{good-lambda:local:P0}
\sigma\big(\{x\in P_0: S_{Q_0} u(x)>\beta\,\lambda, \tilde{N}_{Q_0,*} u(x) \le \gamma\,\lambda\}\big)
\le
C\,\gamma^{2\,\vartheta/\theta}
\sigma(P_0),
\end{equation}
where all the constants are independent of $P_0$.

We recall that $P_0$ is an arbitrary cube in $\{P_j\}_j$, which is a family of Calder\'{o}n-Zygmund cubes associated with $E_\lambda$ for $\lambda>\fint_{Q_0} S_{Q_0} u\,d\sigma$.  In addition,
for $\sigma$-a.e. $x\in Q_0$ such that $S_{Q_0} u(x)>\beta\,\lambda$, we have $\lambda<\beta\,\lambda<S_{Q_0}u(x)\le \mathcal{M}_{Q_0} \left(S_{Q_0}u\right)(x)$
and therefore $x\in E_\lambda$. Using these observations and \eqref{good-lambda:local:P0} in each $P_j$,  we obtain
\begin{align*}
&\sigma\big(\{x\in Q_0: S_{Q_0} u(x)>\beta\,\lambda, \tilde{N}_{Q_0,*} u(x) \le \gamma\,\lambda\}\big)
\\
&
\qquad\qquad=
\sigma\big(\{x\in Q_0: S_{Q_0} u(x)>\beta\,\lambda, \tilde{N}_{Q_0,*} u(x) \le \gamma\,\lambda\}\cap E_\lambda\big)
\\
&
\qquad\qquad=
\sum_j\sigma\big(\{x\in P_j: S_{Q_0} u(x)>\beta\,\lambda, \tilde{N}_{Q_0,*} u(x) \le \gamma\,\lambda\}\big)
\\
&
\qquad\qquad\le
C\,\gamma^{2\,\vartheta/\theta}\sum_j \sigma(P_j)
\\
&
\qquad\qquad=
C\,\gamma^{2\,\vartheta/\theta}\sigma(E_\lambda)
\\
&
\qquad\qquad=
C\,\gamma^{2\,\vartheta/\theta} \sigma\big(\{x\in Q_0: \mathcal{M}_{Q_0}(S_{Q_0} u)(x)>\lambda\}\big).
\end{align*}
Thus, we have shown that for every $\lambda>\fint_{Q_0} S_{Q_0} u\,d\sigma$ we have
\begin{multline}\label{good-lambda:lambda-big}
\sigma\big(\{x\in Q_0: S_{Q_0} u(x)>\beta\,\lambda, \tilde{N}_{Q_0,*} u(x) \le \gamma\,\lambda\}\big)
\\
\le
C\,\gamma^{\theta'} \sigma\big(\{x\in Q_0: \mathcal{M}_{Q_0}(S_{Q_0} u)(x)>\lambda\}\big).
\end{multline}

Let us now consider the case $\lambda\le \fint_{Q_0} S_{Q_0} u\,d\sigma$. We are going to show that
\begin{equation}\label{good-lambda:lambda-small}
\sigma\big(\{x\in Q_0: S_{Q_0} u(x)>\beta\,\lambda, \tilde{N}_{Q_0,*} u(x) \le \gamma\,\lambda\}\big)
\le
C\,\gamma^{\theta'} \sigma(Q_0).
\end{equation}

We repeat the previous computations with $P_0=Q_0$ with the main difference that $P_0$ is no longer a maximal Calder\'{o}n-Zygmund cube. We take the same set $F_\lambda$ and we assume that $\sigma(F_\lambda)>0$, otherwise the desired estimate is trivial.  Before we used the maximality of $P_0$ to show that $\sigma(F_\lambda)<\sigma(P_0)$ for $\beta$ large enough and this was only used to obtain that $\sigma(F)<\sigma(P_0)$. Here we cannot do that and we proceed as follows.
By the inner regularity of $\sigma$ we can find a compact set $\tilde{F}$,  such that $\emptyset\neq \tilde{F}\subset F_\lambda\subset P_0$ and
$$
0<\frac12\,\sigma(F_\lambda)\le \sigma(\tilde{F})\le\sigma(F_\lambda)\le\sigma(P_0).
$$
If $\sigma(F_\lambda)<\sigma(P_0)$ or $\sigma(\tilde{F})<\sigma(F_\lambda)$ we set $F=\tilde{F}$ and we have $\sigma(F)<\sigma(P_0)$. Otherwise, i.e., if
$\sigma(\tilde{F})=\sigma(F_\lambda)=\sigma(P_0)$, we use the following lemma:

\begin{lemma}\label{lemma:dyadic-small-measure}
Let $\Omega$ have the ADR property with constant $C_1$, i.e.,
$$
C_1^{-1}\,r^n\le \sigma(\Delta(x,r))\le C_1\,r^n
$$
for all $x\in\pom$ and $0<r<\diam \pom$. Let $Q\in\dd$ and recall that there exists $\Delta_Q=\Delta(x_Q,r_Q)$ with $r_Q\approx \ell(Q)$ with the property that $\Delta_Q\subset Q\subset \Delta(x_Q,C_2\,r_Q)$ for some uniform constant $C_2\ge 1$. Take $\tau_0=(2\,C_1^2)^{-1/n}$ and set $\Delta_Q^{\tau_0}=\Delta(x_Q,\tau_0\,r_Q)$. Then,
$$
\overline{\Delta_Q^{\tau_0}}\subset \Delta_Q\subset Q
\qquad\mbox{and}\qquad
(2\,C_1^4\,C_2^n)^{-1}\sigma(Q)\le \sigma\big(\overline{\Delta_Q^{\tau_0}}\big)\le \frac34\,\sigma(Q).
$$
\end{lemma}

Assuming this momentarily (the proof is given below), we apply this result to $P_0$. Set $F=\tilde{F}\cap \overline{\Delta_{P_0}^{\tau_0}}\subset \tilde{F}$ which is a compact set.
As we have $\tilde{F}\subset P_0$ and $\sigma(\tilde{F})=\sigma(F_\lambda)=\sigma(P_0)$ we obtain
\begin{multline*}
(2\,C_1^4\,C_2^n)^{-1}\sigma(F_\lambda)
=
(2\,C_1^4\,C_2^n)^{-1}\sigma(P_0)
\le
\sigma\big(\overline{\Delta_{P_0}^{\tau_0}}\big)
=
\sigma\big(\overline{\Delta_{P_0}^{\tau_0}}\cap \tilde{F}\big)
+
\sigma\big(\overline{\Delta_{P_0}^{\tau_0}}\setminus \tilde{F}\big)
\\
\le
\sigma(F)+\sigma(P_0\setminus \tilde{F})
=
\sigma(F)
\le
\sigma\big(\overline{\Delta_{P_0}^{\tau_0}}\big)
\le
\frac34\,\sigma(P_0)
=
\frac34\,\sigma(F_\lambda).
\end{multline*}
Thus, we have found a compact set $F$ such that $\emptyset\neq F\subset F_\lambda\subset P_0$ and
$$
0<(2\,C_1^4\,C_2^n)^{-1}\sigma(F_\lambda)\le\sigma(F)<\sigma(F_\lambda)=\sigma(P_0).
$$
This allows us to run the stopping time argument and find the family $\F$ as before. We then
continue the argument and notice that in \eqref{local-Sf}, we used the maximality of $P_0$.
Here, the analogous estimate is trivial:   since $P_0=Q_0$, it follows that
for any $x\in F\subset F_\lambda$, we have
$S_{P_0}u(x)=S_{Q_0}u(x)>\beta\,\lambda>(\beta-1)\,\lambda$ and
therefore $ F\subset \{x\in P_0: S_{P_0} u(x)>(\beta-1)\,\lambda\}$.
From this point the proof continues without change (except for the fact that we have $\sigma(F_\lambda)\lesssim\sigma(F)$ in place of $\sigma(F_\lambda)\le 2\,\sigma(F)$), so that
\eqref{good-lambda:local:P0} holds, and in the present case
this is our desired estimate \eqref{good-lambda:lambda-small}.

Let $1<q<\infty$ and write $a_{Q_0}=\fint_{Q_0} S_{Q_0} u\,d\sigma<\infty$. Then
\begin{align*}
\|S_{Q_0} u\|_{L^q(Q_0)}^q
&=
\beta^q\,\int_{0}^\infty q\,\lambda^q\sigma\big(\{x\in Q_0: S_{Q_0} u(x)>\beta\,\lambda\}\big)\,\frac{d\lambda}{\lambda}
\\
&\le
\beta^q\,\int_{0}^\infty q\,\lambda^q\sigma\big(\{x\in Q_0: S_{Q_0} u(x)>\beta\,\lambda, \tilde{N}_{Q_0,*} u(x) \le \gamma\,\lambda\}\big)\,\frac{d\lambda}{\lambda}
\\
&\qquad+
(\beta/\gamma)^q\,\|\tilde{N}_{Q_0,*} u\|_{L^q(Q_0)}^q
\\
&
=\beta^q\,\int_0^{a_{Q_0}}\cdots\frac{d\lambda}{\lambda}+\beta^q\,\int_{a_{Q_0}}^\infty\cdots\frac{d\lambda}{\lambda} +(\beta/\gamma)^q\,\|\tilde{N}_{Q_0,*} u\|_{L^q(Q_0)}^q
\\
&
=I+II+(\beta/\gamma)^q\,\|\tilde{N}_{Q_0,*} u\|_{L^q(Q_0)}^q.
\end{align*}
For $I$ we use \eqref{good-lambda:lambda-small} and Jensen's inequality:
$$
I
\le
C\,\gamma^{\theta'} \beta^q\,\sigma(Q_0)\,\int_0^{a_{Q_0}} q\,\lambda^q\,\frac{d\lambda}{\lambda}
=
C\,\gamma^{\theta'} \beta^q\,\sigma(Q_0)\,a_{Q_0}^q
\le
C\,\gamma^{\theta'} \,\beta^q\,\,\|S_{Q_0} u\|_{L^q(Q_0)}^q.
$$
For $II$ we use \eqref{good-lambda:lambda-big} and the fact that $\mathcal{M}_{Q_0}$ is bounded on $L^q(Q_0)$ (notice that $\mathcal{M}_{Q_0}$ is the localized dyadic Hardy-Littlewood maximal function and we can insert the characteristic function of $Q_0$ in the argument for free):
\begin{multline*}
II
\le
C\,\gamma^{\theta'}\, \beta^q\,\int_{a_{Q_0}}^{\infty} q\,\lambda^q\,\sigma\big(\{x\in Q_0: \mathcal{M}_{Q_0}(S_{Q_0} u)(x)>\lambda\}\big)\,\frac{d\lambda}{\lambda}
\\=
C\,\gamma^{\theta'}\, \beta^q\,\|\mathcal{M}_{Q_0}(S_{Q_0} u)\|_{L^q(Q_0)}^q
\le
C\,\gamma^{\theta'}\, \beta^q\,\|S_{Q_0} u\|_{L^q(Q_0)}^q.
\end{multline*}
Gathering the obtained estimates we conclude that
$$
\|S_{Q_0} u\|_{L^q(Q_0)}^q
\le
C\,\gamma^{\theta'}\, \beta^q\,\|S_{Q_0} u\|_{L^q(Q_0)}^q
+
(\beta/\gamma)^q\,\|\tilde{N}_{Q_0,*} u\|_{L^q(Q_0)}^q.
$$
Choosing  $\gamma$  small enough so that $C\,\gamma^{\theta'}\, \beta^q<\frac12$ the first term in the right hand side can be hidden in the left hand side (since it is finite by assumption) and we conclude as desired \eqref{Sf-NT}.
\end{proof}

\begin{proof}[Proof of Lemma \ref{lemma:self-impro-RHp}]
Notice that in proving \eqref{RHP-rest}, we may suppose that the balls $B$ and $B_0$ have respective radii $r_B\ll r_{B_0}$;  otherwise, if $r_B\approx r_{B_0}$, then \eqref{RHP-rest} reduces
immediately to \eqref{eqn:scale-inva:converse-approx:wo-N}.
We now may proceed as in the first part of the proof of Proposition \ref{prop:RHp-scale-N} in order to
use Corollary \ref{cor2.poles}: we take $\Delta=B\cap\Omega$ and
for every  $Y\in\Omega\setminus 2\kappa_0 B$
we have
$$
\fint_{\Delta'} k^Y\,d\sigma
\approx
\omega^Y(\Delta)\,\fint_{\Delta'} k^{X_{\Delta}}\,d\sigma
\qquad B'\subset B.
$$
By the Harnack chain condition this estimate holds for $Y=X_{\Delta_0}$.
Then for $\sigma$-a.e. $x\in\Delta$ we take $B'=B(x,r)$ and let $r\to 0$ to obtain $k^{X_{\Delta_0}}(x)\approx \omega^{X_{\Delta_0}}(\Delta)\,k^{X_{\Delta}}(x)$. Consequently, by \eqref{eqn:scale-inva:converse-approx:wo-N} we have
$$
\left(\fint_{\Delta} \left(k^{X_{\Delta_0}}\right)^{\tilde{p}}\,d\sigma\right)^{\frac1{\tilde{p}}}
\lesssim
\omega^{X_{\Delta_0}}(\Delta)
\,
\left(\fint_{\Delta} \left(k^{X_{\Delta}}\right)^{\tilde{p}}\,d\sigma\right)^{\frac1{\tilde{p}}}
\lesssim
\frac{\omega^{X_{\Delta_0}}(\Delta)}{\sigma(\Delta)}
=
\fint_{\Delta} k^{X_{\Delta_0}}\,d\sigma
.
$$
\end{proof}

\begin{proof}[Proof of Lemma \ref{lemma:dyadic-small-measure}]
The proof is almost trivial. What $\overline{\Delta_Q^{\tau_0}}\subset \Delta_Q$ follows at once from the fact that $\tau_0<1$. On the other hand, let us notice that $\overline{\Delta_Q^{\tau_0}}\subset \Delta(x_Q, (3/2)^{1/n} \,\tau_0\,r_Q)$ and therefore the ADR property gives
$$
\sigma\big(\overline{\Delta_Q^{\tau_0}}\big)
\le
\sigma\big(\Delta(x_Q, (3/2)^{1/n} \,\tau_0\,r_Q))
\le
C_1\,\frac{3}{2}\,\tau_0^n\,r_Q^n
=
\frac34\,C_1^{-1}\,r_Q^n
\le
\frac34\,\sigma(\Delta_Q)
$$
and
$$
\sigma(Q)
\le
\sigma(\Delta(x_Q,C_2\,r_Q))
\le C_1\,C_2^n\,r^n
\le
C_1^2\,C_2^n\,\tau_0^{-n}
\sigma(\Delta_Q^{\tau_0})
\le
2\,C_1^4\,C_2^n\,\sigma\big(\overline{\Delta_Q^{\tau_0}}\big).
$$
\end{proof}

\subsubsection{Good-$\lambda$ inequality for truncated cones}
In order to apply Proposition \ref{prop:good-lambda} we need to know a priori that $\|S_{Q_0} u\|_{L^q(Q_0)}<\infty$ (qualitatively), which can be verified as follows.  Since $\Omega_N$ satisfies the exterior corkscrew condition at small scales, we may invoke \cite{JK}  to control the square function by the non-tangential maximal operator with constants that may depend on $N$ (we recall that at this stage $\Omega=\Omega_N$), but this gives, in particular,  qualitative finiteness of the square function.
The results of \cite{JK} apply to square functions and non-tangential maximal functions defined
via the classical cones $\widetilde{\Gamma}_\alpha (x):=
\{Y \in\Omega : |Y-x|<(1+\alpha)\,\delta(Y)\}$, but our dyadic cones may be compared to the classical cones by varying the aperture parameter $\alpha$.  Thus, we may infer also that the our square function has
(qualitatively) finite $L^p$ norm, given finiteness of our non-tangential maximal function, so
we may then
apply Proposition \ref{prop:good-lambda} to obtain that
\eqref{Sf-NT} holds with uniform bounds.

There is, however, a different approach that consists of working with cones
that are ``truncated'' so that they stay away from the boundary. This in turns implies
that the truncated square function is bounded and therefore that the corresponding
left hand side is finite. Passing to the limit we conclude that the a priori finiteness
hypothesis can be removed from Proposition  \ref{prop:good-lambda}. We present
this argument for the sake of self-containment, and because the argument is of independent interest.

Before stating the precise result we introduce some notation. Given $Q_0$, we take $k\ge k(Q_0)+2 $ (recall that $k(Q_0)$ is defined in such a way $\ell(Q_0)=2^{-k(Q_0)}$) a large enough integer  (eventually, $k\uparrow\infty$). We define $\Gamma_{Q_0}^{k}(x)$  to be the truncated cone where the cubes in the union satisfy  additionally  $\ell(Q)\ge 2^{-k}$. The corresponding truncated square function is written as $S_{Q_0}^{k}$.

\begin{proposition}\label{prop:good-lambda-truncated}
Given $Q_0\in\dd$, let $u$ be harmonic in $2B^{\kappa_0}_{Q_0}\cap\Omega$. Then, for every $1<q<\infty$ and for every $k\ge k(Q_0)+2$ we have
\begin{equation}\label{Sf-NT-trunct}
\|S_{Q_0}^k u\|_{L^q(Q_0)}
\le
C\,\|\tilde{N}_{Q_0,*} u\|_{L^q(Q_0)},
\end{equation}
where the constant $C$ is independent of $k$ and consequently \eqref{Sf-NT} holds whether or not the left hand side is finite.
\end{proposition}

\begin{proof}

It is straightforward to see that the truncated cones are away from the boundary: the cubes that define the cones have side length at least $2^{-k}$ and thus the corresponding Whitney boxes have side length at lest $C\,2^{-k}$, thus $\dist(\Gamma_{Q_0}^k,\partial\Omega)\gtrsim 2^{-k}$. Note that $u$ is harmonic in $2\,B^{\kappa_0}_{Q_0}\cap\Omega$ and therefore smooth in $\frac32\,B^{\kappa_0}_{Q_0}\cap\{Y\in\Omega:\delta(Y)\gtrsim 2^{-k}\}$. Thus,
\eqref{condition:kappa0-cones-sawtooth} implies $|\nabla u|\le C_{Q_0,k,u}$ in $\Gamma_{Q_0}^{k}(x)$, $x\in Q_0$,
and consequently $S_{Q_0}^k u\in L^\infty(Q_0)$. The bound may depend of $k$, $Q_0$ but we will use this qualitatively and a not quantitatively. Note also, that if $Q\in\dd_{Q_0}$ with $\ell(Q)=2^{-k}$ then $\Gamma_{Q_0}^k(x)$ does not depend on $x\in Q$ (i.e. $\Gamma_{Q_0}^k(x)=\Gamma_{Q_0}^k(x')$ for every $x$, $x'\in Q$) and thus $S_{Q_0}^k u$ is constant on $Q$.

Let us write $\mathcal{M}_{Q_0}^k$ for the truncated dyadic maximal function where the cubes in the sup have side length at least $2^{-k}$. Associated to $S_{Q_0}^k u$,  and for each $\lambda>\fint_{Q_0} S_{Q_0}^k u\,d\sigma$ we define the corresponding $E_\lambda$; we clearly have $E_\lambda=\cup_j P_j$ with $P_j$ being maximal cubes as before satisfying further $\ell(P_j)\ge 2^{-k}$ (note that as we have observed that $S_{Q_0}^k u$ is constant on cubes of size $2^{-k}$ we could have taken $\mathcal{M}_{Q_0}$ obtaining the same family $\{P_j\}$ with the same properties.) We also observe that $S_{Q_0}^k u \le \mathcal{M}_{Q_0}^k(S_{Q_0}^k u)$ a.e.~in $Q_0$.
Fix one of these cubes $P_0$  and define $F_\lambda$ with the truncated square function $S_{Q_0}^k$ replacing $S_{Q_0}$ ---we do not truncate the non-tangential maximal operator.  Then we proceed as  before, assume that $\sigma(F_\lambda)>0$ and find the corresponding set $F$, and a family $\F$ with the same properties as before. We can easily see that $\Gamma_{Q_0}^k(x)\subset \Gamma_{P_0}^k(x)\cup \Gamma_{Q_0}^k(z)$ for every
$x\in F$ and $z\in\tilde{P}_0$. Then, since $P_0$ is a Calder\'{o}n-Zygmund cube in $E_\lambda$, we can pick $z_0\in \tilde{P}_0$ with $S_{Q_0}^k u(z)\le \lambda$ (indeed we know that this happens in a set of positive measure in $\tilde{P}_0$.) Then for any
$x\in F\subset F_\lambda \,,$  we have
\begin{align}\label{local-Sf-k}
\beta\,\lambda
&
<
S_{Q_0}^k u(x)
=
\left(\iint_{\Gamma_{Q_0}^k(x)} |\nabla u(Y)|^2\,\frac{dY}{\delta(Y)^{n-1}}\right)^{\frac12}
\\
&
\le
\left(\iint_{\Gamma_{P_0}^k(x)} |\nabla u(Y)|^2\,\frac{dY}{\delta(Y)^{n-1}}\right)^{\frac12}
+
\left(\iint_{\Gamma_{Q_0}^k(z_0)} |\nabla u(Y)|^2\,\frac{dY}{\delta(Y)^{n-1}}\right)^{\frac12}
\nonumber
\\
&=
S_{P_0}^k u(x)+ S_{Q_0}^k u(z_0)
\nonumber
\\
&\le
S_{P_0}^k u(x)+\lambda \nonumber
\end{align}
and therefore $ F\subset \{x\in P_0: S_{P_0}^k u(x)>(\beta-1)\,\lambda\}$.

Next, using that $S_{P_0}^k u(x)\le S_{P_0} u(x)$ we easily obtain \eqref{sf-L2} and from this point the argument goes through without change. Thus, we obtain the following analog of \eqref{good-lambda:local:P0}
\begin{equation}\label{good-lambda:local:P0:truncated}
\sigma(\{x\in P_0: S_{Q_0}^k u(x)>\beta\,\lambda, \tilde{N}_{Q_0,*} u(x) \le \gamma\,\lambda\})
\le
C\,\gamma^{2\,\vartheta/\theta}
\sigma(P_0),
\end{equation}
where all the constants are independent of $P_0$ and $k$.

Let us consider the case $\lambda\le \fint_{Q_0} S_{Q_0}^k u\,d\sigma$. The same argument as before works in this case and we easily obtain
\begin{equation}\label{good-lambda:lambda-small:truncated}
\sigma(\{x\in Q_0: S_{Q_0}^k u(x)>\beta\,\lambda, \tilde{N}_{Q_0,*} u(x) \le \gamma\,\lambda\})
\le
C\,\gamma^{\theta'} \sigma(Q_0).
\end{equation}

Gathering \eqref{good-lambda:local:P0:truncated} and \eqref{good-lambda:lambda-small:truncated} and repeating the computations above we conclude that
$$
\|S_{Q_0}^k u\|_{L^q(Q_0)}^q
\le
C\,\gamma^{\theta'}\, \beta^q\,\|S_{Q_0}^k u\|_{L^q(Q_0)}^q
+
(\beta/\gamma)^q\,\|\tilde{N}_{Q_0,*} u\|_{L^q(Q_0)}^q,
$$
where all the constants are independent of $k$. As observed, the fact that we are working with truncated square functions gives us that $S_{Q_0}^k u\in L^\infty(Q_0)$ (we use this in a qualitative way but not quantitatively). Thus,
choosing  $\gamma$  small enough so that $C\,\gamma^{\theta'}\, \beta^q<\frac12$ the first term in the right hand side can be hidden in the left hand side. Hence, we obtain as desired \eqref{Sf-NT-trunct}
$$
\|S_{Q_0}^k u\|_{L^q(Q_0)}^q
\le
C\, \|\tilde{N}_{Q_0,*} u\|_{L^q(Q_0)}^q
$$
where $C$ is independent of $k$. Letting $k\uparrow \infty$, the monotone convergence theorem gives \eqref{Sf-NT} without assuming that the left hand side is finite. Thus we have proved a version of Proposition \ref{prop:good-lambda} where there is no need to assume that the left hand side is finite.
\end{proof}

\subsection{Step 4: Proof of \eqref{desired-S-Green}}\label{section:proof-S}

In order to obtain \eqref{desired-S-Green} we shall use Proposition \ref{prop:good-lambda}, applied to $u =\nabla G$,
with our exponent $q=\tilde{p}$. Thus, we need to study the non-tangential maximal function of $u$,
and it suffices to do this for each component of $u$, i.e., for any given partial derivative of $G$.   More precisely, we set
$u(Y)=\partial_{Y_j} G(Y,\hat{X}_{Q})$,  where $\hat{X}_{Q}$ is the corkscrew point associated to $\hat{\Delta}_{Q}$  defined in Section \ref{section-step2-Tb}. As mentioned there,  $\hat{X}_{Q}\notin 6\,\tilde{B}_Q$ with $\tilde{B}_Q=B(x_Q, \kappa_1\,\ell(Q))$ and $\kappa_1>\kappa_0$. In particular  $\hat{X}_Q\notin 6\,B^{\kappa_0}_Q$,
so that $u$ is harmonic in $2B^{\kappa_0}_Q\cap\Omega$.

\begin{proposition}\label{prop:NT-Green}
We have
$$
\|\tilde{N}_{Q,*} u\|_{L^{q}(Q)}^q
\le
C\,\sigma(Q)^{1-q}\,.
$$
\end{proposition}

\begin{proof}
Let $x\in Q$ and $Y\in\tilde{\Gamma}_{Q}(x)$. By \eqref{condition:kappa0-cones-sawtooth} applied to $Q$ in place of $Q_0$ we have $\tilde{\Gamma}_{Q}(x)\subset B^{\kappa_0}_{Q}\cap\Omega$ and therefore $u$ is harmonic in $B(Y,3\,\delta(Y)/4)\subset \frac74B^{\kappa_0}_{Q}$. We then obtain by the mean value property of harmonic functions, Caccioppoli's inequality, the Harnack chain condition and Lemma \ref{lemma2.cfms}, that
\begin{multline*}
|u(Y)|
\le
\fint_{B(Y,\delta(Y)/2)} |\nabla_Z G(Z,\hat{X}_{Q})|\,dZ
\lesssim
\delta(Y)^{-1}\,
\left(
\fint_{B(Y,\frac34\delta(Y))} G(Z,\hat{X}_{Q})^2\,dZ
\right)^{\frac12}
\\
\lesssim
\frac{G(Y,\hat{X}_{Q})}{\delta(Y)}
\approx
\frac{G(X_{\Delta_Y},\hat{X}_{Q})}{\delta(Y)}
\approx
\frac{\omega^{\hat{X}_{Q}}( \Delta_Y)}{\delta(Y)^n}
\end{multline*}
where $\Delta_Y=\Delta(y,\delta(Y))$, with $y\in\pom$ such that $|Y-y|=\delta(Y)$. It is easy to see that we can find $C$ large enough so that $x\in\Delta(y, C\,\delta(Y))$ and $\Delta_Y\subset C\,\hat{\Delta}_{Q}$.  As usual, we let $X_{C\,\hat{\Delta}_{Q}}$ denote a Corkscrew point with respect to $C\,\hat{\Delta}_{Q}$.
Then,
\begin{multline*}
|u(Y)|
\lesssim
\frac1{\delta(Y)^n}\,\int_{\Delta_Y} k^{\hat{X}_{Q}}\,d\sigma
\le
\frac1{\delta(Y)^n}\,\int_{\Delta(y,C\,\delta(Y))} k^{\hat{X}_{Q}}\,\chi_{C\,\hat{\Delta}_{Q}}\,d\sigma
\\[4pt]\lesssim \, \frac1{\delta(Y)^n}\,
\int_{\Delta(y,C\,\delta(Y))} k^{X_{C\,\hat{\Delta}_{Q}}}\,\chi_{C\,\hat{\Delta}_{Q}}\,d\sigma
\,\lesssim\,
M(k^{X_{C\,\hat{\Delta}_{Q}}}\,\chi_{C\,\hat{\Delta}_{Q}})(x)\,,
\end{multline*}
where in the next-to-last inequality we have used the Harnack Chain condition.
Therefore,  using that $q=\tilde{p}$ and \eqref{eqn:scale-inva:converse-approx:wo-N}, we obtain
\begin{multline*}
\|\tilde{N}_{Q,*} u\|_{L^{q}(Q)}^q
\,\lesssim\,
\|M(k^{X_{C\,\hat{\Delta}_{Q}}}\,\chi_{C\,\hat{\Delta}_{Q}})\|_{L^{q}(Q)}^q
\,\lesssim\,
\int_{C\,\hat{\Delta}_{Q}} \left(k^{X_{C\,\hat{\Delta}_{Q}}}\right)^q\,d\sigma
\\[4pt]
\lesssim
\sigma(C\,\hat{\Delta}_{Q})^{1-q}
\lesssim
\sigma(Q)^{1-q}.
\end{multline*}
\end{proof}

Now we are ready to establish \eqref{desired-S-Green}. Given $Q\in\dd$, we first apply Proposition \ref{prop:good-lambda-truncated} (or else Proposition \ref{prop:good-lambda}) with $q$, with $Q$ in place of $Q_0$, and with $u = \partial_{Y_j} G(\cdot,\hat{X}_Q)$ as above, since
the latter is harmonic in $2B^{\kappa_0}_Q\cap\Omega$.  Then, using
Proposition \ref{prop:NT-Green}, we obtain as desired that
$$
\|S_{Q} u\|_{L^q(Q)}^q
\le
C\,\|\tilde{N}_{Q,*} u\|_{L^q(Q)}^q
\le
C\,\sigma(Q)^{1-q} \,.
$$

\subsection{UR for $\Omega$}\label{section:UR-general}
To conclude the proof of our main result we see that the UR property for the approximating domains $\Omega_N$ with uniform bounds passes to $\Omega$. As observed before, as a consequence of the $Tb$ theorem Theorem \ref{theor:Tb} we have obtained that \eqref{Tb-conclusion} holds
for $\Omega_N$ with uniform bounds. This in turn implies that $\pom_N$ is UR (see \cite[p. 44]{DS2}) with uniform bounds. To obtain that this property is preserved when passing to the limit we use an argument,
based on ideas of Guy David, along the lines of \cite[Appendix C]{HM-I} (indeed the present situation is easier) where we have to switch the roles of $\Omega$ and $\Omega_N$.

To show that $\Omega$ has the UR property we use the singular integral characterization. We recall that a closed, $n$-dimensional ADR set $E\subset \ree$ is UR if and only if for all singular kernels $K$ as below, and corresponding truncated singular integrals $T_\eps$,  we have that
\begin{equation}\label{UR-SIO}
\sup_{\eps>0}\int_{E} |T_\eps f|^2\,dH^n\leq C_K \int_E|f|^2\,dH^n.
\end{equation}
Here,
$$
T_{E,\eps}f(x)=T_{\eps} f(x):= \int_{E} K_\eps (x-y)\,f(y)\,dH^n(y)\,,
$$
where $K_\eps(x) := K (x)\,\Phi(|x|/\eps)$,  with $0\leq \Phi\leq 1$,
$\Phi(\rho)\equiv 1$ if $ \rho\geq 2,$ $\Phi(\rho) \equiv 0$ if $\rho\leq 1$, and $\Phi \in
C^\infty(\mathbb{R})$, and where
the singular kernel $K $ is an odd function, smooth on $\ree\setminus\{0\}$,
and satisfying
\begin{align}\label{K:size}
|K(x)|\,&\leq\, C\,|x|^{-n}
\\[4pt]\label{grad-K:size}
|\nabla^m K(x)|\,&\leq \,C_m\,|x|^{-n-m}\,,\qquad\forall m=1,2,3,\dots \,.
\end{align}

We also introduce the following extension of these operators
\begin{equation}\label{T:boundary-to-domain}
\T_{E} f(X):= \int_{E} K(X-y)\,f(y)\,dH^n(y)\,,\qquad X \in \ree\setminus E.
\end{equation}
We define non-tangential approach regions $\Upsilon^E_\tau(x)$ as follows.  Let
$\mathcal{W}_E$ denote the collection of cubes in the Whitney decomposition
of $\ree\setminus E$, and set $\mathcal{W}_\tau(x):= \{I\in \mathcal{W}_E: \dist(I,x) <\tau \ell(I)\}$.
We then define
$$
\Upsilon^E_\tau(x):=\bigcup_{I\in \mathcal{W}_\tau(x)} I^*
$$
(thus, roughly speaking, $\tau$ is the ``aperture'' of $\Upsilon^E_\tau(x)$).  For
$F\in \mathcal{C}(\ree\setminus E)$ we
may then also define the non-tangential maximal function
$$
N^E_{*,\tau} (F)(x):= \sup_{Y\in \Upsilon^E_\tau(x)}|F(Y)|.
$$

Let us recall \cite[Lemma C.5]{HM-I} which states that if $E\subset \ree$ is
$n$-dimensional UR, we then have
\begin{equation}\label{Non-tan:UR}
\int_E \left(N^E_{*,\tau}\left(\T_E f\right)\right)^2\, dH^n \,\leq\, C_{\tau,K} \int_E |f|^2 dH^n.
\end{equation}
for every $0<\tau<\infty$ and with $C_{\tau,K}$ depending only on $n$, $\tau$, $K$ and the UR constants.

After these preliminaries, in order to show that $\pom$ is UR we take one of the previous kernels $K$ and form the corresponding operators $T_\eps=T_{\pom, \eps}$. Fix $\eps>0$ and $N\ge 1$ such that $\eps\gg 2^{-N}$. We write $\pom=\cup_j Q_j$ with $Q_j\in\dd_N(\pom)$. For fixed $j$, if we write $\tilde{Q}_j$ for the dyadic parent of $Q_j$, we have by  construction that
$X_{\tilde{Q}_j}\in \interior(U_{\tilde{Q}_j})\subset \Omega_N$. Then in the segment joining $x_j$ (the center of $Q_j$) and $X_{\tilde{Q}_j}$ there exists a point $\hat{x}_j\in \pom_N$. Next we take $Q_j(N)$ the unique cube in $\dd_N(\pom_N)$ such that $\hat{x}_j\in Q_j(N)$. Note that we have $\dist(Q_j, Q_j(N))\approx 2^{-N}$. Since $\pom_N$ is ADR with uniform bounds in $N$, any
given $Q(N)\in \dd(\pom_N)$ can serve in
this way for at most a bounded number of $Q_j\in\dd(\pom)$. Thus we have
\begin{equation}\label{bded-overlap-cubes}
\sum_{Q_j\in\dd_N(\pom)} 1_{Q_j(N)}(x)\leq C\,,\qquad\forall x\in\pom_N.
\end{equation}

As usual, we set $\sigma:=H^n|_{\pom}$, and we now also let
$\sigma_N:=H^n|_{\pom_N}$.  For $\tau$ large enough, we have that if $x'\in Q_j(N)$ and $x\in Q_j$, then $x\in\Upsilon^{\,\pom_N}_\tau(x')$.
Thus,
\begin{multline*}
\int_{\partial\Omega} |\T_{\partial\Omega_N} f|^2\, d\sigma
=
\sum_{Q_j\in\dd_N(\pom)} \int_{Q_j} |\T_{\partial\Omega_N} f|^2\, d\sigma
\\
=
\sum_{Q_j\in\dd_N(\pom)} \frac1{\sigma_N(Q_j(N))}\int_{Q_j(N)}
\int_{Q_j} |\T_{\pom_N} f(x)|^2\, d\sigma(x)\,d\sigma_N(x')
\\
\lesssim \sum_{Q_j\in\dd_N(\pom)} \int_{Q_j(N)}
\left(N^{\pom_N}_{*,\tau}\left(\T_{\pom_N} f\right)\right)^2\, d\sigma_N
 \leq C_{\tau,K} \int_{\pom_N}
|f|^2 d\sigma_N\,,
\end{multline*}
    where in the last line we have used first the ADR properties of $\pom_N$ and
$\pom$, and then \eqref{bded-overlap-cubes} and \eqref{Non-tan:UR} with $E=\pom_N$ (as we may do,
since $\pom_N$ is UR with uniform bounds).
In particular we observe that $C_{\tau,K}$ is independent of $N$.

Thus we have shown that $\T_{\pom_N}:L^2(\pom_N)\to L^2(\partial\Omega)$.  Since the kernel
$K$ is odd, we therefore obtain by duality that
\begin{equation}\label{estimate-cal:T}
\T_{\pom }:L^2(\pom)\to L^2(\pom_N).
\end{equation}
Let us now observe that $K_\eps$ is odd, smooth away the origin and satisfies \eqref{K:size}, \eqref{grad-K:size} uniformly in $\eps$. Thus \eqref{estimate-cal:T} applies to the corresponding operator $\T_{\pom,\eps}$ defined by means of $K_\eps$, that is, we have $\T_{\pom,\eps}:L^2(\pom)\to L^2(\pom_N)$ with bounds that are uniform on $N$ and $\eps$. Then we proceed as in Case 2 of \cite[Appendix C]{HM-I} and write
\begin{align*}
&\int_{\partial\Omega} |T_{\pom,\eps} f|^2\, d\sigma
=
\sum_{Q_j\in\dd_N(\pom)} \int_{Q_j}|T_{\pom,\eps} f|^2\, d\sigma
\\
&
\quad=
\sum_{Q_j\in\dd_N(\pom)} \frac1{\sigma_N(Q_j(N))}\int_{Q_j(N)}
\int_{Q_j}|T_{\pom,\eps} f(x)|^2\, d\sigma(x)\,d\sigma_N(x')
\\
&
\quad\lesssim
\sum_{Q_j\in\dd_N(\pom)} \frac1{\sigma_N(Q_j(N))}\int_{Q_j(N)}
\int_{Q_j}|T_{\pom,\eps} f(x)-\T_{\pom,\eps} f(x')|^2\, d\sigma(x)\,d\sigma_N(x')
\\
&
\qquad\qquad+
\sum_{Q_j\in\dd_N(\pom)} \int_{Q_j(N)}
|\T_{\pom,\eps} f(x')|^2\,d\sigma_N(x'),
\\
&\quad=I+II,
\end{align*}
where we have use the ADR property for both $\pom$ and $\pom_N$. To estimate $II$ we use \eqref{bded-overlap-cubes} and \eqref{estimate-cal:T} applied to $\T_{\pom,\eps}$:
$$
II\lesssim
\int_{\pom_N}|\T_{\pom,\eps} f|^2\,d\sigma_N
\le
C_{\tau,K}\,\int_{\pom}|f|^2\,d\sigma.
$$
On the other hand, standard Calder\'{o}n-Zygmund arguments that are left to the interested reader give that \eqref{K:size} and \eqref{grad-K:size} imply
$$
|T_{\pom,\eps} f(x)-\T_{\pom,\eps} f(x')|
\lesssim
M_{\pom} f(x),
\qquad
x\in Q_j,\ x'\in Q_j(N),
$$
where $M_{\pom}$ is the Hardy-Littlewood maximal operator on $\pom$ and the constants are uniform in $\eps$ and $N$,
and where we have used that $2^{-N}\ll\eps$.
Hence,
$$
I
\lesssim
\sum_{Q_j\in\dd_N(\pom)}
\int_{Q_j} M_{\pom} f(x)^2\,d\sigma(x)
=
\int_{\pom}M_{\pom} f(x)^2\,d\sigma(x)
\lesssim
\int_{\pom}|f(x)|^2\,d\sigma(x).
$$
Gathering our estimates, we conclude that
$$
\int_{\partial\Omega} |T_{\pom,\eps} f|^2\, d\sigma
\lesssim
\int_{\pom}|f|^2\,d\sigma,
$$
where the implicit
constants are independent of $\eps$ and $N$, which gives at once that $\pom$ is UR.\qed

\begin{remark}
Gathering together this section with \cite[Appendix C]{HM-I} we conclude that $\pom$ is UR if and only if $\Omega_N$ is UR with uniform bounds for all $N\gg 1$: here  we have shown the right-to-left implication and the converse argument is given in \cite[Appendix C]{HM-I}.
\end{remark}

\end{document}